\documentclass[12pt]{amsart}
\usepackage{thmtools}
\usepackage{amsfonts,amsmath,amssymb}
\usepackage{hyperref}
\usepackage{paralist}
\usepackage{stmaryrd}
\usepackage[linesnumbered,ruled,vlined,norelsize]{algorithm2e}
\usepackage[utf8]{inputenc}

\declaretheoremstyle[bodyfont=\normalfont]{noncursive}
\declaretheorem{theorem}
\declaretheorem[numberwithin=section]{lemma}

\declaretheorem[numberlike=lemma]{proposition}
\declaretheorem[numberlike=lemma]{corollary}
\declaretheorem[style=noncursive,numberlike=lemma]{definition}

\declaretheorem[style=noncursive,numberlike=lemma]{remark}
\declaretheorem[style=noncursive,numberlike=lemma]{observation}

\declaretheorem[style=noncursive,numberlike=lemma]{convention}

\DeclareMathOperator{\Hom}{Hom}
\DeclareMathOperator{\ord}{ord}

\usepackage{amsmath,amssymb,amsthm,enumerate}

\begin{document}
\numberwithin{equation}{section}

\def\1#1{\overline{#1}}
\def\2#1{\widetilde{#1}}
\def\3#1{\widehat{#1}}
\def\4#1{\mathbb{#1}}
\def\5#1{\frak{#1}}
\def\6#1{{\mathcal{#1}}}

\def\C{{\4C}}
\def\R{{\4R}}
\def\N{{\4N}}
\def\Z{{\4Z}}

\def \im{\text{\rm Im }}
\def \re{\text{\rm Re }}
\def \Char{\text{\rm Char }}
\def \supp{\text{\rm supp }}
\def \codim{\text{\rm codim }}
\def \Ht{\text{\rm ht }}
\def \Dt{\text{\rm dt }}
\def \hO{\widehat{\mathcal O}}
\def \cl{\text{\rm cl }}
\def \bR{\mathbb R}
\def \bC{\mathbb C}
\def \bP{\mathbb P}
\def \C{\mathbb C}
\def \bL{\mathbb L}
\def \bZ{\mathbb Z}
\def \bN{\mathbb N}
\def \scrF{\mathcal F}
\def \scrK{\mathcal K}
\def \scrM{\mathcal M}
\def \cR{\mathcal R}
\def \scrJ{\mathcal J}
\def \scrA{\mathcal A}
\def \scrO{\mathcal O}
\def \scrV{\mathcal V}
\def \scrL{\mathcal L}
\def \scrE{\mathcal E}

\oddsidemargin=0.1in \evensidemargin=0.1in \textwidth=6.4in
\headheight=.2in \headsep=0.1in \textheight=8.4in
\newcommand{\rl}{{\mathbb{R}}}
\newcommand{\cx}{{\mathbb{C}}}
\newcommand{\id}{{\mathbb{I}}}
\newcommand{\m}{{\mathcal{M}}}
\newcommand{\dbar}{\overline{\partial}}
\newcommand{\Db}[1]{\frac{\partial{#1}}{\partial\overline{z}}}
\newcommand{\abs}[1]{\left|{#1}\right|}
\newcommand{\e}{\varepsilon}
\newcommand{\tmop}[1]{\ensuremath{\operatorname{#1}}}
\renewcommand{\Re}{\tmop{Re}}
\renewcommand{\Im}{\tmop{Im}}
\newcommand{\dist}{{\mathrm{dist}}}
\newcommand {\OO}{{\mathcal O}}
\newcommand {\Sig}{{\Sigma}}
\renewcommand {\a}{\alpha}
\newcommand {\MC}{M^{\mathbb C}}
\renewcommand {\b}{\beta}
\newcommand {\Q}{\mathcal Q}
\newcommand{\dop}[1]{\frac{\partial}{\partial #1}}
\newcommand{\dopt}[2]{\frac{\partial #1}{\partial #2}}
\newcommand{\vardop}[3]{\frac{\partial^{|#3|} #1}{\partial {#2}^{#3}}}

\newcommand{\CC}[1]{\mathbb{C}^{#1}}
\newcommand{\CP}[1]{\mathbb{CP}^{#1}}
\newcommand{\RR}[1]{\mathbb{R}^{#1}}
\newcommand{\dw}{\frac{\partial}{\partial w}}
\newcommand{\dz}{\frac{\partial}{\partial z}}
\numberwithin{equation}{section}
\newcommand{\diffcr}[1]{\rm{Diff}_{CR}^{#1}}
\newcommand{\Hol}[1]{\rm{Hol}^{#1}}
\newcommand{\Aut}[1]{\rm{Aut}^{#1}}
\newcommand{\hol}[1]{\mathfrak{hol}^{#1}}
\newcommand{\aut}[1]{\mathfrak{aut}^{#1}}
\newcommand{\fps}[1]{\C\llbracket #1 \rrbracket}
\newcommand{\fpstwo}[2]{#1\llbracket #2 \rrbracket}
\newcommand{\cps}[1]{\C\{#1\}}

\title[]{The equivalence theory for infinite type hypersurfaces in $\CC{2}$}

\author{Peter Ebenfelt}
\address{Department of Mathematics, University of California at San Diego, La Jolla, CA 92093-0112}
\email{pebenfelt@ucsd.edu}

\author{Ilya Kossovskiy}
\address{Department of Mathematics, Masaryk University, Brno, Czechia//
Department of Mathematics, University of Vienna, Vienna, Austria}
\email{kossovskiyi@math.muni.cz}

\author{Bernhard Lamel}
\address{Department of Mathematics, University of Vienna, Vienna, Austria}
\email{bernhard.lamel@univie.ac.at}


\thanks{The first author was supported in part by the NSF grant DMS-1600701. The second author was supported in part by the Czech Grant Agency (GACR) and the Austrian Science Fund (FWF). The third author
was supported in part by the Austrian Science Fund (FWF)}



\begin{abstract}
We develop a classification theory for real-analytic hypersurfaces in $\CC{2}$ in the case when the hypersurface is of  {\em infinite type} at the reference point.
This is the remaining, not yet understood case in $\CC{2}$ in the {\it Problème local}, formulated by  H.\,Poincar\'e in 1907 and asking for a complete biholomorphic classification of real hypersurfaces in complex space. One novel aspect of our results is a notion of {\em smooth normal forms} for real-analytic hypersurfaces. We rely fundamentally on the recently developed CR -- DS technique in CR-geometry.
\end{abstract}

\maketitle   

\def\Label#1{}


\def\cn{{\C^n}}
\def\cnn{{\C^{n'}}}
\def\ocn{\2{\C^n}}
\def\ocnn{\2{\C^{n'}}}


\def\dist{{\rm dist}}
\def\const{{\rm const}}
\def\rk{{\rm rank\,}}
\def\id{{\sf id}}
\def\aut{{\sf aut}}
\def\Aut{{\sf Aut}}
\def\CR{{\rm CR}}
\def\GL{{\sf GL}}
\def\Re{{\sf Re}\,}
\def\Im{{\sf Im}\,}
\def\span{\text{\rm span}}

\def\codim{{\rm codim}}
\def\crd{\dim_{{\rm CR}}}
\def\crc{{\rm codim_{CR}}}

\def\lr{\longrightarrow}
\def\phi{\varphi}
\def\eps{\varepsilon}
\def\d{\partial}
\def\a{\alpha}
\def\b{\beta}
\def\g{\gamma}
\def\G{\Gamma}
\def\D{\Delta}
\def\Om{\Omega}
\def\k{\kappa}
\def\l{\lambda}
\def\L{\Lambda}
\def\z{{\bar z}}
\def\w{{\bar w}}
\def\Z{{\1Z}}
\def\t{\tau}
\def\th{\theta}

\def\H{\hat H}

\emergencystretch15pt
\frenchspacing

\newtheorem{Thm}{Theorem}[section]
\newtheorem{Cor}[Thm]{Corollary}
\newtheorem{Pro}[Thm]{Proposition}
\newtheorem{Lem}[Thm]{Lemma}

\theoremstyle{definition}\newtheorem{Def}[Thm]{Definition}

\theoremstyle{remark}
\newtheorem{Rem}[Thm]{Remark}
\newtheorem{Exa}[Thm]{Example}
\newtheorem{Exs}[Thm]{Examples}

\def\bl{\begin{Lem}}
\def\bl{\begin{Lem}}
\def\el{\end{Lem}}
\def\bp{\begin{Pro}}
\def\ep{\end{Pro}}
\def\bt{\begin{Thm}}
\def\et{\end{Thm}}
\def\bc{\begin{Cor}}
\def\ec{\end{Cor}}
\def\bd{\begin{Def}}
\def\ed{\end{Def}}
\def\br{\begin{Rem}}
\def\er{\end{Rem}}
\def\be{\begin{Exa}}
\def\ee{\end{Exa}}
\def\bpf{\begin{proof}}
\def\epf{\end{proof}}
\def\ben{\begin{enumerate}}
\def\een{\end{enumerate}}
\def\beq{\begin{equation}}
\def\eeq{\end{equation}}

\tableofcontents{}

\section{Introduction}
\subsection{Overview}
In  1907, Poincar\'e \cite{38.0459.02} initiated the study of the
(bi)holomorphically invariant
geometry of real hypersurfaces in complex space. He formulated
the classification problem
 in the following form:

\smallskip

\noindent{\bf {\em Problème local} (H.\,Poincar\'e, 1907).}  Given two germs $(M,p)$ and $(M^*,p^*)$ of real-analytic hypersurfaces in $\CC{2}$, find all local biholomorphic maps $F:\,(\CC{2},p)\mapsto (\CC{2},p^*)$ mapping the hypersurfaces into each other: $F(M)\subset M^*$.

\smallskip

Even though he did not succeed in solving this problem,
Poincar\'e  discovered that germs of real hypersurfaces have local invariants and
possess strong rigidity properties.  He showed that given any real hypersurface $M$, most other real hypersurfaces
 $M^*$ are not holomorphically equivalent to the given one. As a model for real hypersurfaces in $\CC{2}$, he employed {the (unit) sphere $S^3\subset\CC{2}$}. He thought of a generic hypersurface in $\CC{2}$ as a certain perturbation of the model, and showed that  the dimension of the automorphism group of the model is an upper bound for that of a perturbed hypersurface.
This work of Poincar\'e is often considered as a starting point of {Cauchy-Riemann geometry}, or CR-geometry for short, and has inspired intensive research in this subject in the century following its publication. In particular, the original {\it Problème local} of
Poincar\'e mentioned above has attracted a lot of attention starting with E. Cartan's solution \cite{Cartan:1933ux,Cartan:1932ws} of the
problem for strictly pseudoconvex hypersurfaces in $\C^2$. The general problem, however,
has remained open.
%
%

In this paper, we consider the remaining open case in the {\it Problème local} for real
hypersurfaces in a complex euclidean space of dimension $2$, and develop a classification theory for real-analytic hypersurface in $\CC{2}$ at generic {\em infinite type} points. Notably, as a part of the outcome, we introduce a notion of {\em smooth normal forms} for real-analytic hypersurfaces.

Before giving a historical background for the {\it Problème local} and  stating our main results, we should point out that the theory
we develop shares many distinctive traits with the {Poincar\'e-Dulac classification theory} for vector fields. The latter was also initiated by
H.\, Poincar\'e as a way to classify (and analyze the
dynamics of) vector fields at their singular points.
Poincar\'e suggested first bringing the germ of a vector field at a singularity to a {normal form}, known as the {Poincar\'e-Dulac normal form}, and then integrating the vector field in the new, normalized coordinates. For constructing the normal form, Poincar\'e developed the homological method: one compares the Taylor expansion of a given vector field $X$ at a singular point $p$ with  that of its model, namely, the linearization of $X$ at $p$, and attempts to approximate $X$ by the model as closely as possible. By using the homological method, Poincar\'e discovered {\em resonances} --- obstructions for linearizing a vector field, existence of which depends on the the spectrum of the linear part of a vector field at a singularity.  The Poincar\'e-Dulac normal form can be either {\em convergent} (in case the spectrum of the linear part lies in the {Poincar\'e domain}) or divergent (the case of the {Siegel domain}). In the latter case, the normal form is merely {\em formal} and gives only a proper subset of the complete set of invariants, and therefore there are analytic invariants of a singularity supplementing the formal normal form. Existence of such invariants forms the content of the {\em Stokes phenomena} for the classification problem under discussion. For more details here we refer to the excellent  book \cite{MR2363178} of Ilyashenko-Yakovenko on the subject, and to Lombardi-Stolovitch \cite{MR2722512} for some important recent developments. We emphasize that even though the Poincar\'e-Dulac theory provides a satisfactory and in a sense complete solution to the classification problem for singularities of vector fields, some aspects of the holomorphic classification (such as the situation of vanishing linear part at a singularity or various Stokes phenomena)  are still not understood completely. Similar exceptions occur in our treatment of the {\it Problème local} (see Section 1.5 for a related discussion).

\subsection{Historical background of the {\em Problème local}}

The classification problem for real-analytic hypersurfaces in complex space was first considered  by Poincar\'e, and
followed by Cartan's solution in the case of
 Levi-nondegenerate real hypersurfaces
in $\C^2$.
The classification of Levi-nondegenerate hypersurfaces
in $\C^N$, $N\geq 2$,
has been completed by Tanaka \cite{Tanaka:1962ti} and Chern-Moser \cite{CM74}. These works used two principal approaches to the equivalence problem. Tanaka and Chern employed a differential geometric approach, extending and developing  E.\,Cartan's method of equivalence.
This approach relies heavily on a certain uniformity of the
geometric structure under consideration.  In contrast, Moser's
approach in \cite{CM74} used another method to
solve
the classification problem (adapting and developing the
homological method of Poincar\'e), which has been more successfully further adapted
to the Levi-degenerate situation. This approach is via normal
forms: instead of comparing to a geometric model situation,
one tries to find a unique coordinate choice by prescribing
normalization conditions of the defining equation of $M\subset\CC{N}$ in these
coordinates. The key ingredient in the standard application of this approach is finding the right model, for which Moser used {\em real hyperquadrics}
$$\im w=H(z,\bar z)$$
($H$ is a nondegenerate Hermitian form on $\CC{N-1}$). The normal form is then obtained by approximating, in a suitable sense, a  hypersurface by its model as closely as possible.

We shall remark at this point that, in what follows, we will be distinguishing between three different notions of equivalence. We say that two germs
of real-analytic hypersurfaces
$(M,p)$, $(M',p')$ in $\CC{N}$ are {\it formally equivalent} if
there exists a germ of a formal map
$H \in \fps{Z - p}^N $ which satisfies $H(M) \subset M'$ and $H(p) = p'$. We say
that $(M,p)$ and $(M',p')$ are {\it biholomorphically equivalent}
if there exists a germ of a biholomorphism
$H \in \cps{Z - p}^N $ such that $H(M) \subset M'$ and $H(p) = p'$.
And lastly, we say that $(M,p)$ and $(M',p')$ are {\it CR equivalent}
if there exists a smooth CR diffeomorphism $h\colon U\to M'$ defined
in a neighbourhood $U$ of $p$ in $M$ with $h(p)=p'$. Here we recall
 that $h$ is CR if its component functions are smooth CR functions
 on $M$, or equivalently, if the differential
 $dh$ restricts to a complex linear map of the complex tangent spaces $T_p^\C M = T_p M \cap i T_p M$, $dh_p \colon T_p^\C M \to T_p^\C M'$. We note that a biholomorphic equivalence restricts to a CR equivalence. Moreover, a CR equivalence induces a formal equivalence.

Coming back to the normal form problem, we distinguish between {\em formal} and {\em convergent} normal forms. A formal
normal form is given by the construction of formal coordinates $(z^*,w^*)= H(z,w)  \in \fps{z,w}^N $, while
a convergent normal form is given by the construction of (local) holomorphic coordinates
$H(z,w) \in \cps{z,w}^N$. A formal normal form solves the formal equivalence problem, whereas a convergent normal form solves the biholomorphic equivalence problem. Notably, Moser's normal form is {\em convergent}, as is proved in \cite{CM74}.

We next outline the progress on the classification problem in the {\em Levi-degenerate case}, when some (and sometimes all) ingredients of both the Cartan-Tanaka-Chern approach and Moser's approach are not applicable. Let us recall first that a real-analytic
hypersurface $M\subset \C^2$
is {\em of finite type} (at a point $p\in M$) if it does not
contain a holomorphic curve through $p$. (Such a hypersurface also satisfies the Hörmander-Kohn bracket-generating condition at $p$ \cite{BERbook}). A construction of a formal normal
form for finite type (but Levi-degenerate) hypersurfaces in $\C^2$ was carried out
by Kol\'a\v r in \cite{MR2189248}. He employed models of the type
$$\im w=P(z,\bar z),$$ where $P(z,\bar z)$ is a nonzero homogeneous polynomial without harmonic terms of degree $k\geq 3$. Notably, for the class of finite type hypersurfaces, it is known by the work of
 Baouendi-Ebenfelt-Rothschild \cite{BERjams} that every formal holomorphic map
actually converges. Therefore, remarkably, Kol\'a\v r's {\em formal} normal form for
finite type hypersurfaces
provides a solution to the {\em biholomorphic} equivalence problem
for these hypersurfaces. It also provides a solution to the CR equivalence
problem, because every CR diffeomorphism of
two real-analytic hypersurfaces of finite type in $\C^2$ extends
to a biholomorphic map by a result of Baouendi, Jacobowitz, and Treves \cite{MR808223}.
Kol\'a\v r's normal form has been shown  to be convergent under certain geometric conditions, see the work \cite{MR3366850} of Kossovskiy-Zaitsev, but is divergent in general \cite{Kolar12}.
For Levi-degenerate hypersurfaces in $\CC{N},\,N\geq 3$ satisfying certain special conditions (in addition to the Hormander-Kohn bracket-generating condition) normal form constructions were carried out by Ebenfelt \cite{MR1647888,MR1684982}
and by Kossovskiy-Zaitsev  \cite{MR3366850}.
For results on normal forms for real submanifolds  of higher codimension as well as CR-singular submanifolds we refer to Ezhov-Schmalz \cite{MR1601397},
Beloshapka \cite{MR1076927},
Lamel-Stolovitch \cite{lasto},
Moser-Webster \cite{Moser:1983tq},\,
Huang-Yin \cite{HY10,MR3600082},
Gong\,\cite{Gong},
Coffman \cite{MR2727363},
Burcea \cite{MR3508264}.
See also Zaitsev   \cite{MR2956254} for normal forms in the {non-integrable} setting. More references and discussion of the normal form problem can be found in the survey \cite{MR3603888}.

The situation changes dramatically when one
considers {\em infinite type} hypersurfaces in
$\C^2$, that is, hypersurfaces $M\subset\CC{2}$
which contain a complex curve $X \subset M$ through the reference point $p$. We first remark that the automorphism aspect of the {\em Problème local} (i.e,, describing possible automorphism algebras of real-analytic hypersurfaces in $\CC{2}$) in the infinite type setting  was addressed in the paper \cite{nonminimalODE} by Kossovskiy-Shafikov.  However, the equivalence problem in this setting appears to be more difficult, and classification results are available only for some particular classes of infinite type hypersurfaces. Even
the existence of a {\em formal} normal form is only known in a
particular setting (the {$1$-infinite type case}) by work of Ebenfelt-Lamel-Zaitsev \cite{MR3609198}. We shall outline some of the difficulties that arise.
First of all, one of the main difficulties for providing even a formal normal form here is perhaps the absence of polynomial models for the problem. For example, in the special class of infinite type
hypersurfaces of the form
$$\im w=(\re
w)\psi(|z|^2),\quad\psi(0)=0,\,\psi'(0)\neq 0,$$ all of which
contain the complex hypersurface $X=\{w=0\}$, {\em any} polynomial model
has an automorphism (isotropy) group of dimension $2$, while the hypersurface
$\im w=(\re w)\tan\left(\frac{1}{2}\arcsin |z|^2\right)$ has an automorphism group of dimension 5 (see \cite{MR2926319}, \cite{MR3319970}). This fact completely rules out the concept of a model in the sense of Poincar\'e-Moser, and thus the strategy of using polynomial models entirely fails in the infinite type setting. Secondly, and probably even
more importantly, the connection between
different notions of equivalence in the infinite type setting is more subtle and has
only become more clear in the last few years, by
work of Kossovskiy, Lamel and Shafikov: There
exist infinite type hypersurfaces $M$ and $M^*$ in $\C^2$
that are {formally but not biholomorphically equivalent}
\cite{divergence}, and there also exist $M$ and $M^*$ that are {CR equivalent but not biholomorphically equivalent} \cite{nonanalytic}.

On the other hand, it was shown by Kossovskiy-Lamel-Stolovitch \cite{KLS} that in this setting,
every formal equivalence arises as the Taylor series of a CR equivalence. The latter means, first of all, that $M$ and $M^*$ are formally equivalent if and only if they are CR equivalent. Furthermore, it allows for a classification of real hypersurfaces by means {\em  smooth normal forms}. The latter notion is introduced and discussed below.

Before describing the main results of this paper in more detail, we mention that there is a powerful analogy explaining the
distinction between finite type and infinite type hypersurfaces
by comparing with the
situation of regular and singular ordinary differential equations (ODEs):
while formal solutions of regular (analytic) ODEs converge, solutions of singular (analytic) ODEs
might diverge, but often extend to actual (smooth) solutions of the singular
ODEs in sectors.
(Actually, this is more than just an analogy, as we shall discuss in more detail
below.)











\subsection{Main results}
In light of the discussion above concerning the classification at finite type points, we only need to deal with
the classification of germs of Levi-nonflat real-analytic hypersurfaces $M\subset\CC{2}$
considered near a point of {\em infinite} type $p\in
M$. If $M$ is such a hypersurface, there is a unique germ of a
 complex hypersurface (complex curve)
 $X\subset M$ passing through $p$. The complex hypersurface
 $X$ consists
of all infinite type points in $M$ near $p$, it is nonsingular and we will
 also refer to it as the {\em infinite type locus of $M$}.
We say that $(M,p)$ is of {\em generic infinite type} if the canonical
extension of the  Levi
form from $M$ to its {\em complexification}
$M^{\CC{}}\subset\CC{2}\times\overline{\CC{2}}$ \,\,locally vanishes
only on the  complexification
$X^{\CC{}}\subset\CC{2}\times\overline{\CC{2}}$ of $X$. (We
refer the reader to Section 2 for
details). If $M$ is a Levi-nonflat real-analytic hypersurface
with infinite type locus $X$, then $M$ is of
generic infinite type at points $p$ outside of a
 proper (possibly empty) real-analytic subset of $X$. In what follows, we shall consider {\em only} germs $(M,p)$ of generic infinite type, and refer the reader to the discussion in Section 1.5 concerning invariants at infinite type points of { nongeneric} type.

We say that  local holomorphic coordinates $(z,w)$, where $w = u+iv$,   near $p$
are {\em admissible} (for $M$)   if in these coordinates   $p$ becomes the
origin and   $M$ is given by  \begin{equation}\Label{madmissiblereal}\label{madmissiblereal}
v=\frac{1}{2}u^m\left(\epsilon |z|^2+\sum_{k,l\geq 2}h_{kl}(u)z^k\bar
z^l\right)=: h(z,\bar z,u) ,\quad \epsilon=\pm 1 \end{equation} (such admissible
coordinates always exist under the generic infinite type assumption, see \cite{nonminimalODE}); in
particular, in these coordinates $X = \{ w = 0\}$.  The integer $m\geq 1$ is an
important invariant of an infinite type hypersurface; if we
want to be explicit, we are going to say that a hypersurface
$M$ as defined by \eqref{madmissiblereal} is of $m$-infinite type
(see Section 2 for details). For an even $m$, we can
further normalize $\epsilon$ to be equal to $1$, while for an odd $m$,
$\epsilon$ is a biholomorphic invariant. Note that the form
\eqref{madmissiblereal} is stable under the group of dilations
\begin{equation}\Label{dilations}\label{dilations}  z\mapsto \lambda z, \quad w\mapsto \mu w,
\quad \mu^{1-m}=\epsilon|\lambda|^2, \quad \lambda\in\CC{}\setminus\{0\}, \quad
\mu\in\RR{}. \end{equation}
Of particular importance for the normal form construction is the collection of Taylor coefficients
\begin{equation}\label{7jet}
J:=\bigl\{h_{22}(0),\,h_{23}(0),h_{33}(0),\,\,h_{24}(0),\,h_{25}(0),\,h_{34}(0)\bigr\},
\end{equation}
which is a part of the $(m+7)$-jet of $h$ at $0$.

We shall construct a (formal or holomorphic) normal form for hypersurfaces of the form \eqref{madmissiblereal}. By
 this, we mean that we  find
  a choice of (formal or holomorphic) admissible coordinates
  which become essentially {\em unique} by
 requiring more conditions on the defining function $h$
  in \eqref{madmissiblereal}.
  As already noted
above, we exploit a link between mappings of infinite type hypersurfaces
with the theory of singular ODEs. As a result, very analogously to the situation in the Poincar\'e-Dulac theory, we get a distinction between {\em nonresonant} and {\em resonant} hypersurfaces, and between classes of hypersurfaces leading to either {\em convergent}, {\em formal} and {\em smooth} normal forms. (We shall note that the existence of possible resonances for the problem became clear already from the above mentioned work \cite{MR3609198} of Ebenfelt-Lamel-Zaitsev).  Although we need to refer the reader
to Sections 5-6 for a detailed discussion of
the resonancy conditions, we point out here that we will
give explicit  polynomials $p$, defined on the
space of partial $(m+7)$-jets \eqref{7jet} of defining functions $h$ in \eqref{madmissiblereal} at $0$, such that the vanishing of
these polynomials corresponds to the resonant hypersurfaces.

It turns out that for the case $m=1$, the non-resonancy condition alone guarantees that the normal form we construct is {\em convergent} (which
is somehow expected  by the convergence result for CR-maps due to Juhlin-Lamel \cite{JLmz}). Further, the normal form is unique, up to a dilation \eqref{dilations}. Moreover, our normal form for $m=1$  coincides with the {\em formal normal form} obtained by Ebenfelt-Lamel-Zaitsev in \cite{MR3609198}. In this way, we prove the {\em convergence}  of the normal form in \cite{MR3609198}. We point out, however, that the normal forms  in \cite{MR3609198} and in the present paper are obtained in completely different ways.

In contrast, in the non-resonant case where $m>1$, we construct a formal
normal form, which is in general {\em
divergent}.  By the result of Kossovskiy-Shafikov referenced above,
we can {not} expect a convergent normal form here.
However, somewhat surprisingly, our formal normal form can be used to define a distinguished {\em smooth normal form} for each hypersurface under consideration. The smooth normal form of $M$ is a smooth nonminimal hypersurface $N$, which is real-analytic and Levi-nondegenerate outside the complex hypersurface $X=\{w=0\}\subset N$ through $0$. The smooth normal form is unique, up to a dilation \eqref{dilations} and the action of the vector field \eqref{wmdw} below. The tool that makes it possible to  distinguish  a canonical {\em smooth} representative of a formal normal form (within the infinite-dimensional space of such possible representatives) is the {\em multi-summability property} of the formal normalizing transformation; we refer to Section 2.6 in \cite{KLS} and also references therein for the details of this concept. Briefly, the algebra of $(k_1,...,k_s)$ multi-summable in a direction $d$  power series (where $k_j\in\mathbb N$, and $d\subset\CC{}$ is a ray originating at $0$) is a certain algebra of formal power series in the variable $w\in\CC{}$ each of which is the Taylor series at $0$ of a {distinguished} holomorphic function in a "large" sector $S$ bisected by $d$ and originating at the origin. The latter algebra can be generalized to the algebra of formal power series $f(z,w)$ that are holomorphic in the variable $z$ and are $(k_1,...,k_s)$ multi-summable in a direction $d$ in the variable $w$ (see Section 2.3 below), and this already allows us to define multi-summability of formal transformations. Braaksma's Fundamental Theorem \cite{braaksma} can be combined then with the formal normal form construction below and used to conclude that the formal normalizing transformation  in fact belongs to the above class of multi-summable transformations. Furthermore, the direction $d$, the "large" sector $S$ and the multi-summability orders $k_1,...,k_s$ can be read of uniquely from the partial $(m+7)$-jet $J$, in \eqref{7jet}, of the defining function of a hypersurface. (In fact, the collection of derivatives $J$ precisely corresponds to the principal matrix in Braaksma's theorem). Finally, as discussed in Section 2.3, the canonical sectorial realizations of multi-summable transformations can be used to produce the canonical smooth normal forms.     We shall also note here that the original idea of using multi-summability for obtaining smooth realizations of formal CR-maps goes back to the work \cite{KLS} of Kossovskiy-Lamel-Stolovitch.

We now formulate our main results in detail. Let us introduce, for each $m\geq 1$, the space $\mathcal N_m$ of (formal) power series $h(z,\bar z,u)$, as in \eqref{madmissiblereal}, satisfying, in addition,
\begin{equation}\Label{nspace1}\label{nspace1}
h_{22}'(u)=h_{23}'(u)=h_{32}'(u)\equiv 0, \,\,h_{33}^{(j)}(0)=0,\,j\notin\{0,m-1\}.
\end{equation}
In other words, we have
\begin{equation}\Label{constants}\label{constants}
h_{22}(u)=h_{22}(0), \quad h_{23}(u)=h_{23}(0),\quad h_{32}(u)=h_{32}(0), \quad
 h_{33}(u)=r+s u^{m-1},
\end{equation}
where $r,s\in\RR{}$ are  constants.

\begin{definition}\Label{innf}\label{innf}
We say that a hypersurface $M$ of the form \eqref{madmissiblereal} is in {\em normal form} if its defining function $v=h(z,\bar z,u)$  satisfies \eqref{nspace1}.
\end{definition}

Obviously, the dilations \eqref{dilations} preserve the property of being in normal form.
We now formulate our normalization results in the cases $m=1$ and $m>1$, respectively.

\begin{theorem}\Label{theor1}\label{theor1}
Let $M\subset\CC{2}$ be a real-analytic Levi-nonflat hypersurface, which is of
 $1$-infinite type at a point $p\in M$, and assume that
  that $p$ is a non-resonant point for $M$.
 Then there exists a biholomorphic transformation
\begin{equation}\Label{F}\label{F}
H:\,(\CC{2},p)\mapsto (\CC{2},0)
\end{equation}
bringing $M$ into  normal form. The normalizing transformation $H$ is uniquely determined by the restriction of its differential $dH|_p$ to the tangent space $T^{1,0}_p M$.
\end{theorem}

\begin{remark} We remark that the normal form in Theorem \ref{theor1} solves the holomorphic equivalence problem in the following way: {\em Two hypersurfaces $M,M^*$  as above  are biholomorphically equivalent at their points $p,p^*$ respectively,  if and only if in some (and hence any) normal forms  at these points there exists a dilation \eqref{dilations} transforming the normal form of one into the other.}

\end{remark}

A typical application of any normal form is a bound
on the
dimension of the isotropy algebras  of the
hypersurfaces under consideration, which we also have here.

\begin{corollary}\Label{cor1}\label{cor1}
For a hypersurface $(M,p)$ satisfying the assumptions of
\autoref{theor1}, we have $\dim_\R \mathfrak{aut} (M,p)\leq 2$.
\end{corollary}

Another application that our normal form allows for is the unique determination of a CR-map between hypersurfaces by its finite jet at a point.

\begin{corollary}\Label{cor1}\label{cor11}
Let $H:\,(M,p)\mapsto (M^*,p^*)$ be a local biholomorphism of two hypersurfaces  satisfying the assumptions of
\autoref{theor1}. Then $H$ is uniquely determined by its $1$-jet at the point $p$.
\end{corollary}
Note that \autoref{cor11} gives a specific bound for
the jet order needed for determination, thereby
improving   the general finite jet determination result of Ebenfelt-Lamel-Zaitsev in \cite{MR1995799} in
this specific case.

For $m>1$, we first give the general formal normal form.

\begin{theorem}\label{theor2}
Let $M\subset\CC{2}$ be a real-analytic Levi-nonflat hypersurface, which is of
$m$-infinite type at a point $p\in M$ for some $m \geq 2$, and assume
 that $p$ is a non-resonant point for $M$.  Then there exists a formal invertible transformation
\begin{equation}\Label{F2}\label{F2}
\widehat H:\,(\CC{2},p)\mapsto (\CC{2},0)
\end{equation}
bringing $M$ into  normal form. The normalizing transformation $\widehat H$ is uniquely determined by the restriction of its differential $d\widehat H|_p$ to the tangent space $T^{1,0}_p M$ and the transverse $m$-th order derivative of the transverse component of $\widehat H$.
\end{theorem}

\begin{remark}\Label{3param}\label{3param} In a coordinate system $(z,w)$
where $p=0$ and $T^{1,0}_p=\{w=0\}$, the normalizing transformation $H=(F,G)$ in \autoref{theor2} is uniquely determined by the complex parameter $\lambda:=F_z(0,0)$ and by the real parameter
\begin{equation}\Label{realparam}\label{realparam}
\tau:=\frac{1}{m!}\frac{\partial^{m} G}{\partial w^{m}}(0,0).
\end{equation}
Alternatively, a normalizing transformation is determined uniquely, up to the right action of the group of dilations \eqref{dilations} and the flow of the vector field
\begin{equation}\Label{wmdw}\label{wmdw}
w^m\frac{\partial}{\partial w}.
\end{equation}
\end{remark}

As discussed above, the assertion of \autoref{theor2} can be significantly strengthened.
\begin{theorem}\label{theor3}
Let $M\subset\CC{2}$ be a real-analytic Levi-nonflat hypersurface, which is of
$m$-infinite type at a point $p\in M$ for some $m \geq 2$, and assume
 that $p$ is a non-resonant point for $M$.  Then there exists a canonical choice of smooth CR-functions $F,G$ on $M$, $F(p)=G(p)=0$, such that the smooth map $H=(F,G)$ is invertible at $p$ and transforms $M$ into a smooth nonminimal hypersurface $N\subset\CC{2}$ the formal expansion of which at the origin is in normal form \eqref{nspace1}.  The normalizing transformation $H\colon M \to N$ is uniquely determined by the restriction of its differential $dH|_p$ to the tangent space $T^{1,0}_p M$ and the transverse $m$-th order derivative of the transverse component of $H$.

\end{theorem}

\begin{remark}\label{canonical} The smooth hypersurface $N$ is called {\em a smooth normal form} of $M$ at $p$. The canonical choice, as mentioned above, of the smooth transformation $H$ in \autoref{theor3} (and hence the smooth normal form $N$) is accomplished by using the $(k_1,...,k_s)$ multi-summability of the formal normalizing transformation $\widehat H$ in a direction $d$, and choosing subsequently the unique sectorial representative $H$ of $\widehat H$. The multi-summability levels  $(k_1,...,k_s)$ and the direction $d$ are uniquely determined by the partial $(m+7)$-jet $J$, as in \eqref{7jet}, of the defining function $h$ at $0$ (see also \autoref{canonical2}). That is, the procedure of assigning to a hypersurface its smooth normal form is well defined within each class of hypersurfaces with fixed data \eqref{7jet} satisfying the conditions of \autoref{theor3}.
\end{remark}

\begin{remark}\label{rem:equivalence} \autoref{theor3} solves the smooth CR-equivalence problem: {\it Two hypersurfaces as above are $C^\infty$ CR-equivalent if and only if their smooth (or, equivalently, formal) normal forms coincide, for some choice of the normal forms.}
\end{remark}

\begin{corollary}\Label{cor2}\label{cor2}
For a hypersurface $(M,p)$ satisfying the assumptions of
\autoref{theor2}, we have $\dim_\R \mathfrak{aut} (M,p)\leq 3$.
\end{corollary}

\begin{corollary}\Label{cor1}\label{cor22}
Let $H:\,(M,p)\mapsto (M^*,p^*)$ be a formal biholomorphism of two hypersurfaces  satisfying the assumptions of
\autoref{theor1}. Then $H$ is uniquely determined by its $m$-jet at the point $p$.
\end{corollary}
\autoref{cor22} gives a specific jet bound for
 the general jet determination result of Ebenfelt-Lamel-Zaitsev \cite{MR1995799} in this particular setting.

\smallskip

Finally, we note that even for $m>1$ it is possible to distinguish a class of real hypersurfaces \eqref{madmissiblereal} ({\em Fuchsian type hypersurfaces}) with a {\em convergent} normal form. This class of hypersurfaces is discussed in detail in Section 8.

\subsection{Principal method} The main tool of the paper is
the recent
CR \,$\lr$\,DS (Cauchy-Riemann manifolds \,$\lr$\,\,Dynamical Systems) technique developed by  Kossovskiy and Lamel in the works \cite{divergence,nonminimalODE,nonanalytic,analytic,KLS} (partially in their joint work with Shafikov and Stolovitch). The technique involves
replacing a given degenerate CR-submanifold $M$  by an appropriate holomorphic dynamical system
$\mathcal E(M)$, and then study mappings of CR-submanifolds
through mappings
of the associated dynamical system.
The approach to replace a real-analytic
CR-manifold by a complex dynamical system is based on the
fundamental parallel between CR-geometry and the geometry of
completely integrable PDE systems. This parallel was first observed by E.~Cartan and
Segre \cite{cartan,segre} (see also Webster \cite{webster}), and was revisited, modernized and further developed in the work of
Sukhov \cite{sukhov1,sukhov2}. The connection
between a CR-manifold and the associated PDE system is via the Segre
family of the CR-manifold. Unlike in the nondegenerate setting in the cited work \cite{cartan,segre,sukhov1,sukhov2}, the CR\,-\,DS technique  deals systematically with the {\em degenerate} setting, providing a dictionary of sorts between CR-geometry and dynamical systems.

In this paper, we give a rigourous formal description
of this process in terms of categories of real hypersurfaces
$\mathfrak{R}_m^\pm$, Segre varieties $\mathfrak{S}_m^\pm$,
and singular holomorphic ODEs $\mathfrak{E}_m$. We shall use
the language of category theory for this. We obtain faithful functors
$\mathfrak{R}_m^\pm \hookrightarrow \mathfrak{S}_m^\pm \hookrightarrow
 \mathfrak{E}_m$. After completing the formal
 classification of holomorphic singular ODEs in $\mathfrak{E}_m$
 satisfying a nonresonance condition,
 we will show that a suitable choice of normal form descends along this
 chain and gives rise to a normal form for hypersurfaces that way.

The paper is organized as follows. In Section 2, we provide preliminaries. In Section 3, we explain how the initial equivalence problem can be (partially) reduced to that for the {associated ODEs}. In Section 4, we explain our general approach to the equivalence problem, and demonstrate that bringing an ODE to a normal form amounts to solving a precise system of $4$ singular second order ODEs with a meromorphic singularity.  In Section 5, we provide a natural normal form for the ODEs under consideration in the case $m=1$, and in Section 6 we provide a family of normal forms in the quite different general $m>1$ case. In Section 7, we complete the last step in solving the equivalence problem in the latter two cases, by establishing a criterion for extracting the fact of equivalence of hypersurfaces from that for associated ODEs. This allows for the desired (respectively holomorphic or smooth) normal forms for infinite type hypersurfaces. Lastly, in Section 8 we treat the Fuchsian type case mentioned above.

\subsection{Concluding remark: Exceptional infinite type points.}  We shall emphasize that our classification theory does {\em not} close the equivalence problem for real hypersurfaces in $\CC{2}$ entirely. Describing the Stokes phenomenon for the case $m>1$ in detail is still open, as is describing biholomorphic invariants for nongeneric (exceptional) infinite type points. Note that similar exceptions occur also in the Poincar\'e-Dulac theory, where certain Stokes phenomena are still not well understood, and the classification of degenerate singularities of vector fields (which can be seen as analogues of nongeneric infinite type points) is understood only in very particular situations. Geometric invariants of the latter degenerate singularities can be  understood completely with the help of  certain blow-up procedures (see e.g. \cite{MR2363178} and references therein). Very similarly,  for exceptional infinite type points in real hypersurfaces,  a blow-up procedure was suggested by Lamel-Mir in \cite{MR2328892}. This procedure reduces,  in a certain sense, the study of holomorphic invariants at exceptional infinite type points to that of generic ones. Thus, the results of the present paper can be used, {in principle}, to describe invariants at exceptional infinite type points as well.





\section{Preliminaries}

\subsection{Infinite type real hypersurfaces}\label{ss:nonminimal} 
We recall that if  $M\subset \CC{2}$ is a real-analytic  hypersurface, then
for any $p\in M$ there exist so-called {\em normal coordinates} $(z,w)$ centered
at $p$ for $M$. The coordinates
being normal means here that $(z,w)$ is
a local holomorphic coordinate system near $p$ in
which $p=0$ and 
$M$ is defined by an equation of the form
\[ v = F (z, \bar z , u),\]
where $w=u+iv$ and $F$ is a germ of a holomorphic function on $\CC{3}$ that satisfies
the conditions
\[ F(z,0,u) = F(0,\bar z, u) = 0\]
and $F(z,\bar z,u) \in \RR{} $ for $(z,u)\in \CC{} \times\RR{}$ close to $0$ (see e.g. \cite{BERbook}). Equivalently, $v = F(z,\bar z,u)$ defines a real hypersurface,
and the Segre varieties
$Q_p=Q_{(0,u)}$ at $p=(0,u)$ are given by $\left\{ w = u \right\}$.

We also recall that a real-analytic hypersurface $M$ is {\em of infinite type} at $p$ if there exists a germ of a nontrivial
complex curve $X\subset M$ through $p$. In normal coordinates $(z,w)$, such
a curve $X$ is necessarily defined by $w = 0$ (because it must be contained in the Segre variety $Q_{0} = \{ w= 0\}$); 
in particular, $X$ is nonsingular. Thus, $M$ is of infinite type at $p=0$ if and only if $F(z,\bar z,0)=0$. Moreover, $M$ is Levi-flat if and only if $F(z,\bar z,u)$ is identically $0$.
Consequently, if $M$ is of infinite type and Levi-nonflat, then it
 can defined by an equation of the form
\begin{equation}\label{mnonminimal}
v=u^m\psi(z,\bar z,u),
\end{equation}
where $m\geq 1$ and 
\begin{equation}\label{mnonminimal}
\psi(z,0,u) = \psi(0,\bar z, u)= 0,\quad \psi(z,\bar z,0)\not\equiv 0.
\end{equation}
The integer $m\geq 1$ is independent of both
the choice of $p\in X$ and also of the choice of
normal coordinates for $M$ at $p$  (see \cite{meylan}), and we
 say that
$M$ is {\em $m$-infinite type} along $X$ (or at $p$).

We are going to utilize several different ways of writing a defining function.
Throughout this paper, we use the {\em complex defining function} $\Theta$ in
which $M$ is defined by  \[  w = \Theta (z,\bar z,\bar w) ;\] it is obtained
from $F$ by solving the equation  \[ \frac{w - \bar w}{2i} = F \left(z,\bar z,
\frac{w+\bar w}{2} \right) \] for $w$, and it agrees with the function $h$
defining the  Segre varieties in those coordinates, that is,  $Q_Z$ is parametrized by  $z\mapsto (z,
\Theta(z,\bar Z))$ for $z$ sufficently close to $0$. We are going to make
extensive use of the Segre varieties and refer the reader to
\cite{BERbook} for a discussion
of their properties in the general case,
and to \cite{nonanalytic} for
specific properties in the infinite type setting.

The complex defining function (in
normal coordinates) satisfies the
conditions  \[ \Theta(z,0,\tau) = \Theta (0, \chi, \tau ) = \tau, \quad
\Theta(z,\chi,\bar \Theta (\chi, z, w)) = w. \] If $M$ is of
$m$-infinite type at $p$, then
$\Theta (z,\chi,\tau) = \tau \theta (z,\chi,\tau)$ and  thus $M$ is defined by
the equation 
$$ w = \bar w \theta(z,\bar z,\bar w) = \bar w  +\bar w^{m}\tilde
\theta(z,\bar z,\bar w),$$
 where $\tilde \theta $ satisfies
$ \tilde \theta (z,0,\tau) = \tilde
\theta (0,\chi, \tau) = 0$  and   $ \tilde\theta (z,\chi,0)\neq 0$.

We also note  that the {complexification} $M^{\CC{}}$ of $M$, which
is the hypersurface in $\CC{2}\times\overline{\CC{2}}$ defined by
$M^{\CC{}} = \left\{ (Z, \zeta) \in U\times \bar U \colon Z \in Q_{\bar \zeta} \right\} $,
is conveniently defined as the graph of the complex defining function $\Theta$, i.e.
$w= \Theta(z,\chi,\tau).$


We also introduce the (real) line
\begin{equation}\label{Gamma}
\Gamma = \{(z,w)\in M \colon z = 0 \} = \{(0,u) \colon u\in\RR{}\}\subset M,
\end{equation}
and recall that
\[ Q_{(0,u)} = \{(z,w)\colon w = u\}\]
for $(0,u)\in \Gamma $. This  property, as already mentioned, is
actually equivalent to the normality of
the coordinates $(z,w)$. Moreover, for any real-analytic curve
$\gamma$ through $p$ one can find normal coordinates $(z,w)$ in
which $\gamma$ (near $p$) is mapped to
 $\Gamma$, as defined in in \eqref{Gamma}. The reader is referred to \cite{BERbook} for more details.

We finally observe that  a real-analytic Levi-nonflat hypersurface
$M\subset\CC{2}$ has infinite type points of two kinds,
which we will
refer to as  {\em generic} and {\em exceptional} infinite type points,
respectively. A generic point $p\in M$ is characterized by the condition that
the complexified Levi form of $M$ only degenerates
on the complexified infinite type locus $w=\tau=0$ near $p$. (The
complexified Levi form is defined similarly to the classical Levi form, but
instead the $(1,0)$ and the $(0,1)$ vector fields are considered on the
complexification $M^{\CC{}}$, see e.g. \cite{BERbook}). We refer to a
non-generic point  $p$ as {\em exceptional}. We note that the
set of exceptional points is a proper (possibly empty) real-analytic subvariety of $X$ and
that $p\in X$ is generic if and only if the Levi-determinant of $M$ vanishes
to order $m$ along {\em any} real curve $\gamma$ through $p$ and transverse to $X$ at $p$.

A generic infinite type point is characterized in normal coordinates by
requiring in addition to \eqref{mnonminimal}  the
condition $\psi_{z\bar z}(0,0,0,)\neq 0.$ If $p$ is a generic
infinite type point, we can further
simplify $M$ to the form \eqref{madmissiblereal} above, or alternatively to the
{\em exponential form}
\begin{equation}\Label{exponential}\label{exponential}  w=\bar w e^{i\bar
w^{m-1}\varphi(z,\bar z,\bar w)},\quad\mbox{where}\quad\varphi(z,\bar z,\bar
w)=\pm z\bar z+\sum_{k,l\geq 2}\varphi_{kl}(\bar w)z^k{\bar z}^l \end{equation}
(see, e.g., \cite{nonminimalODE}).

\subsection{Real hypersurfaces and second order differential equations.}\label{sub:realhyp2ndorderequ}
There is a natural way to associate to
a  Levi nondegenerate real hypersurface
$M\subset\CC{N}$ a system of second order
holomorphic PDEs with $1$ dependent and $N-1$ independent
variables by using
the Segre family of the hypersurface $M$.
This remarkable construction
 goes back to
E.~Cartan \cite{cartan} and Segre \cite{segre}
 (see also a remark by Webster \cite{webster}),
and was recently revisited in the work of Sukhov
\cite{sukhov1,sukhov2} in the nondegenerate setting, and in the work of Kossovskiy, Lamel and Shafikov  in the degenerate setting
(see\cite{divergence,nonminimalODE,nonanalytic,analytic}). 
For the convenience of the reader, we recall this procedure in the case
$N=2$, but refer to the above references for more details.

We assume that $M\subset\CC{2}$ is
a smooth real-analytic hypersurface passing through the origin
 and $U = U^z \times U^w$ is a sufficiently small neighborhood of the origin.
 The  second order holomorphic ODE associated to $M$
 is uniquely determined by the condition that for every
 $\zeta \in U$,  the
 function $h(z,\zeta) = w(z)$ defining the Segre variety
 $Q_\zeta$ as a graph is a solution of this ODE.

To be more precise, one can show that the Levi-nondegeneracy
of  $M$ (at $0$) implies that
near the origin, the Segre map
$\zeta\mapsto Q_\zeta$ is injective and the Segre family has
the so-called transversality property: if two distinct Segre
varieties intersect at a point $q\in U$, then their intersection
at $q$ is transverse (actually it turns out
that, again due to the Levi-nondegeneracy of $M$,
the Segre varieties passing through a point $p$
are uniquely determined by their tangent spaces $T_p Q_\zeta$).
Thus, $\{Q_\zeta\}_{\zeta\in U}$ is a
2-parameter  family of holomorphic
curves in $U$ with the transversality property, depending
holomorphically on $\bar\zeta$. It follows from
the holomorphic version of the fundamental ODE theorem (see, e.g.,
\cite{MR2363178}) that there exists a unique second order
holomorphic ODE $w''=\Phi(z,w,w')$ such that for each $\zeta \in U$,
$w(z) = h(z, \bar \zeta)$ is one of its solutions.

We can carry out the construction of this ODE concretely by
utilizing the complex defining equation $w=\Theta(z,\chi,\tau)$
introduced above.
Recall that  the Segre
variety $Q_\zeta$ of a point $\zeta=(a,b)\in U$ is  now given
as the graph
\begin{equation} \label{segre0}w (z)=\rho(z,\bar a,\bar b). \end{equation}
Differentiating \eqref{segre0} once, we obtain
\begin{equation}\label{segreder} w'=\rho_z(z,\bar a,\bar b). \end{equation}
The system of equations  \eqref{segre0} and \eqref{segreder}
can be solved, using the implicit function theorem, for
 $\bar a$ and $\bar b$. This gives us holomorphic functions
 $A $ and $ B$ such that
$$
\bar a=A(z,w,w'),\,\bar b=B(z,w,w').
$$
The application of the implicit function theorem is possible
since the
Jacobian of the system consisting of  \eqref{segre0} and \eqref{segreder} with
respect to $\bar a$ and $\bar b$ is just the Levi determinant of $M$
for $(z,w)\in M$ (\cite{BERbook}). Differentiating \eqref{segreder} once more,
we can substitute  $\bar a = A(z,w,w') $ and $ \bar b=B(z,w,w')$ to obtain
\begin{equation}\label{segreder2}
w''=\rho_{zz}(z,A(z,w,w'),B(z,w,w'))=:\Phi(z,w,w').
\end{equation}
Now \eqref{segreder2} is a holomorphic second order ODE, for which all
of the functions $w(z) = h(z, \zeta)$ are solutions by construction.
We will denote this associated second order ODE by
$\mathcal E = \mathcal{E}(M) $.

More generally it is possible to  associate  a completely integrable PDE
to any of a wide range of
CR-submanifolds (see \cite{sukhov1,sukhov2}) such that the
correspondence $M\to \mathcal E(M)$ has the following fundamental
properties:

\begin{compactenum}[(1)]

\item Every local holomorphic equivalence $F:\, (M,0)\to (M',0)$
between CR-submanifolds is an equivalence between the
corresponding PDE systems $\mathcal E(M),\mathcal E(M')$;

\item The complexification of the infinitesimal automorphism algebra
$\mathfrak{hol}^\omega(M,0)$ of $M$ at the origin coincides with the Lie
symmetry algebra  of the associated PDE system $\mathcal E(M)$
(see, e.g., \cite{olver} for the details of the concept).
\end{compactenum}

In contrast to the case of a finite type real hypersurface described above,
if
$M\subset\CC{2}$ is of infinite type at the origin one cannot associate
to $M$ a second order ODE or even a more general PDE system near
the origin such that the Segre varieties are graphs of solutions.
However, in \cite{nonminimalODE} and \cite{analytic},
Kossovskiy, Lamel and Shafikov
found an injective correspondence associating to a  hypersurface
$M\subset\CC{2}$ at a  generic infinite type point a
certain {\em singular} complex ODEs $\mathcal E(M)$ with an
isolated  singularity at the origin.
We are going to base our normal form construction fundamentally on
this construction (more details are given in Section 3).

We finally point out that at {\em exceptional} infinite type points, one
can still associate
 a system of singular complex ODEs
to a real-analytic hypersurface $M\subset\CC{2}$ (although possibly of
 higher order $k\geq 2$) as in the paper \cite{KLS}
 Kossovskiy-Lamel-Stolovitch.

\subsection{Complex differential equations with an isolated singularity and sectorial mappings of hypersurfaces} \label{ss:isolated}\mbox{}
A {\em meromorphic ODE with an isolated singularity} is a complex ODE of the kind \begin{equation}\label{mer}
w^m\frac{dy}{dw}=F(w,y).
\end{equation}
Here $m\geq 1$ is an integer, $y=(y_1,...,y_n)\in\CC{n}$, $w$ runs an open neighborhood of the origin in $\CC{}$, and $F(x,y)$ is holomorphic near the origin in $\CC{n+1}$. One is mainly concerned with the dynamics of solutions of \eqref{mer} as $w$ approaches the singularity $w=0$. This dynamics is somewhat parallel to that for the vector field
$$w^m\frac{\partial}{\partial w}+F(w,y)\frac{\partial}{\partial y}.$$
A particularly well studied case here is when the ODE \eqref{mer} is {\em linear}, i.e., $F(w,y)=B(w)y$. In the latter case, the behavior of solutions mainly depends on the integer $m$. For $m=1$, the singularity is called {\em Fuchsian}. In the Fuchsian case, solution demonstrate moderate growth near $w=0$ (i.e., one has the bound   $||y(w)||\leq C|w|^a$  for $w\in S,$ where $S$ is an arbitrary (bounded) sector in the complex plane with the vertex at the origin and the constants $a,C$ depend on $S$). Besides, formal solutions $\widehat y(w)=\sum c_kw^k$ for a Fuchsian ODE are necessarily {\em convergent}. When, otherwise, $m>1$, the situation changes dramatically: solutions do not have a moderate growth in sectors in general, and formal solutions are  {\em divergent} in general.

The behavior of solutions  in the nonlinear case is somewhat analogous. For $m=1$ (when an ODE \eqref{mer} is referred to as a {\em Briott-Bouquett ODE}), formal solutions are necessarily convergent.   Also, notably, a Briot-Bouquet type ODE whose  principal matrix $F_y(0,0)$
has no positive integer eigenvalues, necessarily  has at least one
holomorphic solution (see \cite{laine}). On the other hand, for $m>1$ such formal solutions need not exist, and if exist are divergent in general. In order to understand the nature of divergent solutions, Poincar\'e introduced the notion of {\em asymptotic expansion} for a function $f(w)$ holomorphic in a sector $S$ as above: $\widehat f(w)\sim\sum c_kw^k$. The latter means that, for each $k$, one has the bound $$\left|f(w)-\sum_{j=0}^k c_jw^j\right|=O(|w|^{k+1}), \quad w\rightarrow 0,\,\,w\in S.$$ Poincar\'e showed that, under certain assumptions, formal solutions of ODEs \eqref{mer} with $m>1$ necessarily realize  asymptotic true solutions in sectors. Notably, such solutions with a given asymptotic expansion can  {\em vary} depending on the sector; the latter is the fundamental reason for the {\em Stokes phenomenon} in Dynamical Systems, see e.g. \cite{MR2363178}, \cite{vazow}. It has been a big goal for decades  to establish the sectorial realization of formal solutions in full generality, and also to establish classes of sectorial solutions within which the Borel map (assigning to a function its Taylor series at the vertex) is {\em injective} (that is, one searches here for a {\em unique, canonical} sectorial realization of a formal solution). The latter has been accomplished in the work of a lergae group of mathematicians (an incomplete list here includes Ecalle, Malgrange, Martines, Ramis, Sibuya, Balser, Braaksma), and the respective class of formal series here is referred to as the class of {\em $(k_1,...,k_s)$ multi-summable series in a given direction} discussed briefly in the Introduction. For details and further properties we refer to e.g. \cite{KLS}, \cite{braaksma} and references therein.   Braaksma's Fundamental Theorem \cite{braaksma} establish the multi-summability in all but a few exceptional directions for formal solutions of {\em arbitrary} ODEs \eqref{mer}.
Further information on the classification of isolated
singularities  can be found
in e.g.  \cite{MR2363178}, \cite{vazow},\cite{laine},\cite{braaksma}.

In the end of this section, we shall explain how the concepts of asymptotic expansion and multi-summability can be used for studying CR-maps. We consider classes of functions $f(z,w)$ holomorphic in {\em sectorial domains}, by which we mean domains of the kind $\Delta\times S^\pm\subset\CC{2}$,  where $\Delta$ is a disc centered at $0$ and $S^\pm\subset\CC{}$ are bounded non-intersecting sectors with the vertex at $0$ containing the rays $\RR{\pm}$ respectively.  We further assume that $f(z,w)$ is representable as $R(z,u_1(w),...,u_l(w))$, where $R$ is holomorphic in all its variables in a neighborhood of the origin, and each $u_j(w)$ admits (coincident) asymptotic expansions in $S^\pm$. (One can consider an a bit more general class of functions here,see e.g. \cite{analytic}, but we do not need these details). If now $M$ is a real hypersurface \eqref{mnonminimal}, then in view of $\mbox{Arg}\,w\rightarrow 0$ as $(z,w)\rightarrow 0,\,(z,w)\in M$, it is easy to see that $M\setminus X,\,X=\{w=0\}$ stays inside a sectorial domain when $(z,w)$ are close to the origin, and $f(z,w)$ defines a smooth ($C^\infty$) CR-function on $M$. It is further easy to see that a pair $H=(f,g)$ of such functions with nonzero Jacobian at $0$ defines a smooth CR-map onto a smooth real hypersurface $N\subset\CC{2}$, which also contains $X$ and is real-analytic and Levi-nondegenerate in $M\setminus X$. Such a map $H$ is called {\em a sectorial map of a real hypersurface \eqref{mnonminimal}}.

We finally note that if, in the above construction, the functions $u_j(w)$ is the representation of $f(z,w)$ are sectorial representatives  of $(k_1,...,k_s)$ multi-summable formal power series, then the Borel map within the class of such functions $f$ is also injective.

\section{Reduction to the classification problem for ODEs}
We consider a real-analytic hypersurface with defining
equation as in  \eqref{madmissiblereal}.
The complex defining function of such a hypersurface
is given by
\begin{equation}\Label{madmissiblecomp}\label{madmissiblecomp}
w=\bar w+i \bar w^m\left(\epsilon |z|^2+\sum_{k,\ell\geq 2}\Theta_{k\ell}(\bar w)z^k\bar z^\ell\right).
\end{equation}
We recall from \autoref{ss:nonminimal} that
this means that  the Segre family
 $\mathcal S=\{Q_{(\xi,\eta)}\}$ of $M$
 is given by:
\begin{equation}\Label{mc}\label{mc}
w=\bar\eta e^{i\bar\eta^{m-1}\varphi(z,\bar\xi,\bar\eta)},\quad\mbox{where}
\quad\varphi(z,\bar\xi,\bar\eta)=\epsilon z\bar\xi+\sum_{k,\ell\geq 2}\varphi_{k\ell}(\bar\eta)z^k{\bar\xi}^\ell
\end{equation}
We still think about  \eqref{mc} as a {\em parameterized family of planar complex curves}, depending on the parameters $\xi,\eta$  anti-holomorphically). Differently to the case of
the Segre family of a Levi-nondegenerate hypersurface,
this parameterized family does not satisfy the transversality
condition. We can therefore not expect it to satisfy a
regular ODE, and we will recall later that we can find
a singular ODE which it satisfies.

Our first step however is going to be an observation
which will allow us to restrict the class of maps
we need to consider in our equivalence problem by a fair amount,
since formal transformations between
hypersurfaces of the form we are interested satisfy a
number of restrictions on the low order terms in their
Taylor expansion.
\begin{lemma}\Label{specialmap}\label{specialmap}
Let  $H(z,w) =  \bigl(F(z,w),G(z,w)\bigr)$ be a formal transformation
vanishing at the origin, with invertible Jacobian $H'(0)$, which maps  a hypersurface defined by \eqref{madmissiblereal}
or equivalently \eqref{mc} into another such hypersurface. Then
$H$ satisfies
\begin{equation}\Label{specialg}\label{specialg}
\begin{gathered}
F_z(0,0)=\lambda, \quad  G_w(0,0)=\mu, \quad  G= O(w), \\ G_z=O(w^{m+1}),\quad \mu^{1-m}=|\lambda|^2, \quad \lambda\in\CC{}\setminus\{0\}, \, \mu\in\RR{}.
\end{gathered}
\end{equation}
In addition, we have
\begin{equation}\Label{taureal}\label{taureal}
G_{w^\ell}(0,0)\in\RR{}, \quad \text{ for } \ell \leq m.
\end{equation}
\end{lemma}

\begin{proof} The content of the first part of the
Lemma, \eqref{specialg} is a
well-known consequence of being in normal coordinates.
Indeed, since in normal coordinates, we have $S_{0} = \{ w = 0\}$,
and $H(S_0) \subset S_0'$, necessarily $G(z,0) = 0$, and so
$G(z,w) = O(w)$. Since $H'(0)$ is invertible, we therefore necessarily
have $F_z (0,0) G_w (0,0) \neq 0$, and we can writ $F(z,w) = \lambda z + \dots $ and $G(z,w) = \mu w + \dots $ with $\lambda \mu \neq 0$.

Since $H$ maps a hypersurface
of the form  \eqref{madmissiblecomp}  into another such,
we have
\begin{equation}
	\label{e:basic1} G(z,\bar w+\bar w^{m}\cdot z\bar z \cdot O(1))=
\bar G (\bar z,\bar w)+\bar G(\bar z,\bar w)^m  \bar F(\bar z,\bar w)  F( z, w)\cdot O(1)
\end{equation}
for all $z$, $\bar z$, and $\bar w$.
Evaluating  \eqref{e:basic1} for $ \bar z = 0$ gives
$ G(z, \bar w ) = \bar G (0, \bar w) + O(\bar w^{m+1}) $ and
therefore $G_z (0, \bar w) = O(\bar w^{m+1})$ as  claimed
in \eqref{specialg},  and we also obtain the claim  \eqref{taureal}.

If we compare the coefficient of $z \bar z \bar w^m$ on both sides
of \eqref{e:basic1}, we obtain that $\mu  = \mu^m |\lambda|^2$ and
therefore the missing claim in \eqref{specialg}.
\end{proof}

\autoref{specialmap}  implies in particular
that any transformation $H$ between hypersurfaces defined
by equations of the form \eqref{madmissiblereal} can be factored as
$$H=H_0\circ \psi,$$
for some dilation $\psi$ of the form \eqref{dilations}
and where   $H_0$ is
a transformation of the form:
$$z\mapsto z+f(z,w), \quad w\mapsto w+wg_0(w)+w^{m}g(z,w)$$
with
\begin{equation}\Label{normalmap}\label{normalmap}
 f_z(0,0)=0,\quad   g_0(0)=0,\quad g(z,w)=O(zw), \quad g_0^{(\ell)}(0)\in\RR{}, \quad \ell \leq m-1.
\end{equation}
(for $m=1$ the last condition is void).
In fact, one can also represent $H$ as
$$H=\psi\circ H_0$$
(with a different $H_0$).
We therefore consider the classification problem only under transformations \eqref{normalmap}.

We now recall that  \cite{nonminimalODE,analytic} showed that
we can associate to a hypersurface in the
form \eqref{madmissiblereal} a second order
singular holomorphic ODE $\mathcal E(M)$ given by
\beq\Label{ODE}\label{ODE}
w''=w^m\Phi\left(z,w,\frac{w'}{w^m}\right),
\eeq
where $\Phi(z,w,\zeta)$ is holomorphic near the origin in $\CC{3}$, and satisfies $\Phi=O(\zeta^2)$.
This ODE is characterized by the condition that
any of the functions $w (z) = \Theta (z, \xi,\eta)$,
for $(\xi,\eta ) \in \bar U$, is a solution of the
 ODE \eqref{ODE}). We will decompose  $\Phi$ as
\begin{equation}\Label{expandPhi1}\label{expandPhi1}
\Phi(z,w,\zeta)=\sum_{j,k\geq 0,\ell\geq2} \Phi_{jk\ell}z^kw^j\zeta^\ell.
\end{equation}

The key step in our normalization procedure for hypersurfaces
\eqref{madmissiblereal} is the reduction of the classification
problem of hypersurfaces to the classification of ODEs of the form
\eqref{ODE} which have some additional properties (derived from
being associated to a real hypersurface), which we will carry out  in Section 5.
However, we start by considering
 general parameterized families of planar complex
curves which are given by equations of the  kind \eqref{mc}, instead of Segre families of hypersurfaces
\eqref{madmissiblereal}. If no risk of confusion arises, we
will also call such a parametrized family of planar curves
a { Segre family}.
\begin{definition}
\Label{admissible}\label{admissible}
A parameterized family of planar (formal or holomorphic) complex curves (a Segre family), given by
$\mathcal{S}_\varphi = \{ w = \bar \eta e^{i \bar \eta^{m-1} \varphi(z, \bar \xi, \bar \eta) } \} $ is said to be {\em $m$-admissible} if $\varphi$ satisfies
$\varphi(z,\bar\xi,\bar\eta)=\epsilon z\bar\xi+\sum_{k,\ell\geq 2}\varphi_{k\ell}(\bar\eta)z^k{\bar\xi}^\ell$ with $\epsilon=\pm 1$. Depending on
the sign of $\epsilon$, we say that $\mathcal{S}_\varphi$ is {\em positive} ($\epsilon = +1$) or {\em negative} ($\epsilon = -1$). We call the equation
$w = \bar \eta e^{i \bar \eta^{m-1} \varphi(z, \xi, \eta) }$ the defining
equation of the Segre family.
\end{definition}

We will be dealing with a number of ways of associating objects
 and maps with one another. Even though there is
 some additional structure in our setting, we will not use it, and  it will be sufficient to use
 (elementary) language of category theory.
 For a given  $m\in\N$,  we  introduce the category
$\mathfrak{E}_m$ whose objects are ODEs of the form
\eqref{ODE}; that is, as a set, the set of
objects of $\mathfrak{E}$ can be identified with
power series (formal or convergent) $\Phi  \in \fps{z, w, \zeta}$ of
the form \eqref{expandPhi1}. If $m$ is fixed, we
will usually drop it from the notation; we
fix an arbitrary $m\geq 1$ for now.
 If  $ \mathcal{E}_1 , \mathcal{E}_2 \in \mathfrak{E}$, then $H\in \Hom ( \mathcal{E}_1 , \mathcal{E}_2)$
is given by an invertible formal map $H(z,w) = (F(z,w), G(z,w)) $, where
$F,G \in \fps{z,w}$  transforming graphs of solutions of $\mathcal{E}_1$
into graphs of solutions of $\mathcal{E}_2$ and satisfying \eqref{normalmap}.

For any given $m\in\N$, we also introduce the categories $\mathfrak{S}^+_m$ and $\mathfrak{S}^-_m$
of positive and negative Segre families, respectively, whose objects
are positive (resp. negative) Segre families as defined in \autoref{admissible}. We will again
drop the $m$ from the notation if it is fixed.
The objects of $\mathfrak{S}^+$ and $\mathfrak{S}^-$ can therefore
be identified with power series $\varphi$ as in \eqref{mc}, with $\varepsilon = 1$ or $\varepsilon = -1 $, respectively. A transformation
$\mathcal{H} \in \Hom( \mathcal{S}_1 , \mathcal{S}_2)$
 is an invertible
formal map $\mathcal{H} (z,w,\xi,\eta) = (F(z,w), G(z,w), \Lambda (\xi,\eta), \Omega (\xi, \eta)) $ where $F, G \in \fps{z,w}$ and $\Lambda, \Omega \in \fps{\xi,\eta}$ which satisfies that it maps the family
$\mathcal{S}_1$ into $\mathcal{S}_2$, and which satisfies the normalization
conditions \eqref{normalmap} for $(F,G)$ and the analogous conditions for
$(\Lambda, \Omega)$.

\begin{proposition}\Label{isomorphism}\label{isomorphism}
The categories $\mathfrak{E}$ and $\mathfrak{S}^{\pm}$
are equivalent, that is, there
exists a functor $\mathcal{E} \colon \mathfrak{S}^{\pm} \to \mathfrak{E}$ and
a functor $\mathcal{S}^{\pm} \colon \mathfrak{E} \to \mathfrak{S}^{\pm}$
such that $\mathcal{E} \circ\mathcal{S}^{\pm} = \id_{\mathfrak{E}} $ and
 $ \mathcal{S}^{\pm} \circ \mathcal{E} = \id_{\mathfrak{S}^{\pm}} $.
 The functors satisfy the following conditions:
\begin{compactenum}[\rm i)]
\item $\mathcal{E} (\mathcal S_0) $ is determined by the condition
that  the defining function
 of $\mathcal{S}_0$ satisfies $\mathcal{E} (\mathcal S_0) $,
 and  $\mathcal{S}^{\pm} (\mathcal E_0)$ is a positive Segre family
 all of whose elements satisfy $\mathcal{E}_0$.

\item Convergent elements and morphisms are transformed into convergent
 elements and morphisms via $\mathcal{E} $, $\mathcal{S}^{\pm}$, and $\mathcal{S}^-$.
\end{compactenum}
\end{proposition}

The functors constructed in \autoref{isomorphism} allow us in
particular to translate
 the equivalence problem for complexifications of hypersurfaces
 \eqref{madmissiblereal} under the transformation group \eqref{normalmap} into
 the  classification of  the class of ODEs \eqref{ODE} under transformations \eqref{normalmap}.

\begin{proof}[Proof of \autoref{isomorphism}]
We shall carry out the proof for $\mathfrak{S}^+$, the
case of negative Segre families being completely analogous.

We first show that, if the  equation
\[ \mathcal{E}_0  \colon w''=w^m\Phi\left(z,w,\frac{w'}{w^m}\right) \]
is given, then condition i)
uniquely determines a positive segre family $\mathcal{S}^+ (\mathcal{E}_0)$.
Indeed, substituting the defining equation of a positive Segre family
$\mathcal{S}^+ = \{ w = \bar\eta e^{i\bar\eta^{m-1}\varphi(z,\bar\xi,\bar\eta)} \} $ into the equation $\mathcal{E}_0$ yields  the following (formal or holomorphic) ODE for $\varphi$:
\begin{equation}\Label{assoc}\label{assoc}
\varphi_{zz}=-ie^{i(m-1)\bar\eta^{m-1}\varphi}\cdot\Phi\bigl(z,\bar\eta e^{i\bar\eta^{m-1}\varphi},ie^{i(1-m)\bar\eta^{m-1}\varphi}\varphi_z\bigr)-\bar\eta^{m-1}(\varphi_z)^2.
\end{equation}
We let $\varphi$ denote
 the unique solution of the (resp. formal or holomorphic) ODE \eqref{assoc}
 with  Cauchy data
\begin{equation}\Label{data}\label{data}
\varphi(0,\bar\xi,\bar\eta)=0,\quad \varphi_z(0,\bar\xi,\bar\eta)=\bar\xi.
\end{equation}
We claim that $\varphi$ satisfies furthermore
\begin{equation}\Label{data1}\label{data1}
\varphi(z,0,\bar\eta)=0,\quad \varphi_{\bar\xi}(z,0,\bar\eta)=z,
\end{equation}
so that $\mathcal{S}^+ (\mathcal{E}_0) \in \mathfrak{S}^+$  determines an $m$-admissible family  as required.

To prove the claim,
let us recall that $\Phi(z,w,\zeta)=O(\zeta^2)$.
Hence the substitution $\varphi(z,\bar\xi,\bar\eta)=\bar\xi\psi(z,\bar\xi,\bar\eta)$ turns \eqref{assoc} into a (resp. formal or holomorphic) ODE for $\psi$. If $\psi$ is its solution with the initial data
$$\psi(0,\bar\xi,\bar\eta)=0, \quad \psi_z(0,\bar\xi,\bar\eta)=1,$$
then by uniqueness we have $\varphi=\bar\xi\psi$, which proves the first identity in \eqref{data1}. Substituting again $\varphi=\bar\xi\psi$ into \eqref{assoc} we see that $\varphi_{zz}$ is divisible by $\bar\xi$, which implies the second condition in \eqref{data1}.

The proof that one can find an  ODE for a positive $m$-admissible family  is very analogous to the proof of the same fact for Segre families of real hypersurfaces, which is in detail carried out in  \cite{analytic}),
and the reader can easily make the changes needed to
accomodate the slightly more general situation considered here.
This proves (i).

We now prove (ii). We fix two $m$-admissible families $\mathcal S_1,\mathcal S_2 \in \mathfrak{S}^+$
and their associated ODEs $\mathcal E_1,\mathcal E_2 \in \mathfrak{E}$.
If there is a (resp. formal or holomorphic) map
$(z,w,\xi,\eta) \mapsto (F(z,w), G(z,w), \Lambda(\xi,\eta) , \Omega(\xi,\eta)\in \Hom(\mathcal S_1,\mathcal S_2)$,
then the map $(z,w) \mapsto (F(z,w), G(z,w)) \in \Hom(\mathcal{E}_1, \mathcal{E}_2) $
by definition.

Conversely, if a map
$(z,w)\mapsto H(z,w) = \bigl(F(z,w),G(z,w)\bigr) \in \Hom (\mathcal E_1,\mathcal E_2)$,
let us consider (yet undefined) maps $(\Lambda (\xi, \eta) , \Omega (\xi,\eta))$ and the
the preimage $\tilde{ \mathcal{S}}$ of $\mathcal S_2$ under the map
$\mathcal{H} (z,w,\xi,\eta) = \left(F(z,w),G(z,w) ,\Lambda (\xi, \eta) , \Omega (\xi,\eta) \right) $.

If we denote the defining function of  $\mathcal{S}_2$ by
 $\rho_2$, then the equation $w = \rho(z, \xi,\eta)$ defining $\tilde{\mathcal{S}}$ is given by
 solving the equation
\begin{equation}
	\label{e:pullback}  G(z,\rho(z, \bar \xi,\bar \eta)) = \rho_2 \left( F(z, \rho(z, \bar \xi,\bar \eta)) ,
	\bar \Lambda (\bar \xi, \bar \eta ), \bar \Omega ( \bar \xi, \bar \eta ) \right)
\end{equation}
for $\rho$. The conditions \eqref{normalmap} ensure that this is possible using
 the implicit function theorem.

If we require that $\rho(0,\bar \xi, \bar \eta) = \bar \eta$ and $\rho_z (0, \bar \xi, \bar \eta) = i \bar \xi \bar \eta^m$, then we obtain the following system of equations:
\begin{equation}
	\label{e:sys1} \begin{aligned}
	G(0, \bar \eta) &= \rho_2 \left( F(0,\eta), \bar \Lambda (\bar \xi, \bar \eta) , \bar \Omega (\bar \xi, \bar \eta) \right) \\
	G_z ( 0, \bar \eta) + G_w (0, \bar \eta) i \bar \xi \bar \eta^m &= \frac{\partial \rho_2}{\partial z} \left( F(0,\eta), \bar \Lambda (\bar \xi, \bar \eta) , \bar \Omega (\bar \xi, \bar \eta) \right) \left(F_z ( 0, \bar \eta) + F_w (0, \bar \eta) i \bar \xi \bar \eta^m  \right).
	\end{aligned}
\end{equation}
Using \eqref{normalmap} we
write $G(0,w) = w g_0 (w)$,  $G_z (0, w) = w^m g(w) $  and
set $\Omega(\xi,\eta) = \eta \tilde \Omega (\xi, \eta)$.
By Definition~\ref{admissible} we can write
 $\rho_2(z,\bar \xi , \bar \eta) = \bar \eta \tilde \rho_2 (z,\bar \xi, \bar \eta)$ and
$\rho_{2,z} (z, \bar \xi, \bar \eta) = i\bar \xi \bar \eta^m \sigma (z, \bar \xi, \bar \eta) $. Note
that $g_0 (0) = 1$, $\tilde \rho_2 (0,0,0) = 1$, and $\sigma (0,0,0) = 1$ by
\eqref{normalmap} and Definition~\ref{admissible}.
The system can therefore be rewritten as
\begin{equation}
	\label{e:sys2} \begin{aligned}
	g_0 ( \bar \eta) &= \tilde \rho_2 \left( F(0,\eta), \bar \Lambda (\bar \xi, \bar \eta) , \bar \eta \bar{\tilde \Omega} (\bar \xi, \bar \eta) \right)\bar{\tilde \Omega} (\bar \xi, \bar \eta)  \\
	\frac{g(\bar \eta) + ( g_0 (\bar \eta) + \bar \eta g_0' (\bar \eta)) i \bar \xi }{\left(F_z ( 0, \bar \eta) + F_w (0, \bar \eta) i \bar \xi \bar \eta^m  \right)} &= \sigma  \left( F(0,\eta), \bar \Lambda (\bar \xi, \bar \eta) , \bar \eta \bar{\tilde \Omega} (\bar \xi, \bar \eta) \right) i
	\bar \Lambda (\bar \xi, \bar \eta) \bar{\tilde\Omega} (\bar \xi, \bar \eta)^m.
	\end{aligned}
\end{equation}
The implicit function theorem now ensures that
\eqref{e:sys2} has a unique solution $(\bar \Lambda (\bar \xi, \bar \eta) , \bar{\tilde \Omega} (\bar \xi, \bar \eta))$ satisfying $\bar \Lambda (0,0) = 0$ and $\bar{\tilde \Omega} (0,0) = 1$.
One checks also that \eqref{e:sys2} implies that
 $(\Lambda, \Omega)$ satisfies the normalization
conditions \eqref{normalmap}.

By the first part of the proof, the defining equation $w = \rho_1 (z, \bar \xi, \bar \eta)$ for the Segre family $\mathcal{S}_1$
is uniquely determined by the Cauchy data $\rho_1 (0, \bar \xi, \bar \eta) $
and $\rho_{1,z} (0, \bar \xi, \bar \eta)$ from which we conclude that actually
$\rho = \rho_1$, and so $\mathcal{H} \in \Hom (\mathcal{S}_1 , \mathcal{S}_2)$.

\end{proof}

Lastly, if we define the category $\mathfrak{R}_m^\pm$ of $m$-infinite type hypersurfaces
for which the fixed coordinate system $(z,w)$ is admissible, and where for
$M,M' \in \mathfrak{R}_m^\pm$ we define $\Hom(M,M')$ to consist of all germs of
holomorphic (or all formal) maps $h\colon M \to M'$, then
the associated ODE and the associated Segre family give rise to  faithful
functors $\mathfrak{R}_m^\pm \hookrightarrow \mathfrak{E}_m$ and
$\mathfrak{R}_m^\pm \hookrightarrow\mathfrak{S}_m^\pm$.

\section{The general approach to the classification of ODEs}

\subsection{Reduction to a Cauchy problem}
Based on \autoref{isomorphism}, we now proceed with the holomorphic classification of ODEs \eqref{ODE}. According to the classical formulas for prolonging maps $\CC{}\mapsto\CC{}$ to the space of $2$-jets \cite{bluman},  for two given ODEs $\mathcal E=\left\{w''=\Psi(z,w,w')\right\}$ and $\mathcal E^*=\left\{w''=\Psi^*(z,w,w')\right\}$ and a diffeomorphism $(z,w)\mapsto(\tilde f,\tilde g)$ transforming $\mathcal E$ to $\mathcal E^*$ we have the following basic identity:
\begin{multline}\Label{trule}\label{trule}
\Psi(z,w,w')=\frac{1}{J}\bigg((\tilde f_z+w'\tilde f_w)^3\Psi^*\left(\tilde f(z,w),\tilde g(z,w),\frac{\tilde g_z+w'\tilde g_w}{\tilde f_z+w'\tilde f_w}\right)+\\
I_0(z,w)+I_1(z,w)w'+I_2(z,w)(w')^2+I_3(z,w)(w')^3\bigg),
\end{multline}
where $J:=\tilde f_z\tilde g_w-\tilde f_w\tilde g_z$ is the Jacobian determinant of the transformation and
\beq\begin{aligned}
I_0 &=\tilde g_z\tilde f_{zz}-\tilde f_z\tilde g_{zz}\\
I_1 &=\tilde g_w\tilde f_{zz}-\tilde f_w\tilde g_{zz}-2\tilde f_z\tilde g_{zw}+2\tilde g_z\tilde f_{zw}\\
I_2 &=\tilde g_{z}\tilde f_{ww}-\tilde f_z\tilde g_{ww}-2\tilde f_w\tilde g_{zw}+2\tilde g_w\tilde f_{zw}\\
I_3 &=\tilde g_w\tilde f_{ww}-\tilde f_w\tilde g_{ww}.
\end{aligned}\eeq
Similarly, for a pair of singular ODEs of the kind \eqref{ODE} we get:
\begin{multline}\Label{trule2}\label{trule2}
\Phi\left(z,w,\frac{w'}{w^m}\right)=\frac{1}{J}\bigg((\tilde f_z+w'\tilde f_w)^3\tilde g^m\Phi^*\left(\tilde f(z,w),\tilde g(z,w),\frac{\tilde g_z+w'\tilde g_w}{\tilde g^m(\tilde f_z+w'\tilde f_w)}\right)+\\
I_0(z,w)+I_1(z,w)w'+I_2(z,w)(w')^2+I_3(z,w)(w')^3\bigg),
\end{multline}
where the expressions $J,I_0,..,I_3$ are identical to the ones in \eqref{trule}. In what follows we treat the ODE $\mathcal E$ as the target and $\mathcal E^*$ as the initial ODE, respectively.

Recall that we are looking for  transformations of the kind \eqref{normalmap}.
Rewriting $\tilde f(z,w) = z + f(z,w) $, and $\tilde g (z,w) = w + w g_0 (w) + w^m g(z,w)  $,
 \eqref{trule2} (after dividing by $w^m$) becomes:
\begin{multline}\Label{trule3}\label{trule3}
\Phi\left(z,w,\zeta\right)=\frac{1}{J}\Bigl[\bigl(1+ f_z+w^m f_w\cdot\zeta)^3 (1+g_0(w)+w^{m-1}g\bigr)^m\cdot\\
\cdot \Phi^*\Bigl(z+ f,w+wg_0(w)+ w^mg,\frac{g_z+\zeta(1+ wg_0'+g_0+mw^{m-1}g+w^mg_w)}{(1+g_0(w)+w^{m-1}g)^m(1+f_z+w^m\zeta f_w)}\Bigr)+\\
+ I_0(z,w)+I_1(z,w)\zeta+I_2(z,w)w^{m}\zeta^2+I_3(z,w)w^{2m}\zeta^3\Bigr],
\end{multline}
where $\zeta:=\frac{w'}{w^m}$ and
\beq\Label{2jet}\label{2jet} \begin{aligned}
J&=(1+f_z)(1+wg_0'+g_0+mw^{m-1}g+w^mg_w)-w^mf_wg_z,\\
I_0 &= g_zf_{zz}- (1+f_z)g_{zz},\\
I_1 &=\bigl(1+wg_0'+g_0+mw^{m-1}g+w^mg_w\bigr)f_{zz}-w^mf_wg_{zz}-\\
&-2(1+ f_z)
(mw^{m-1}g_z+w^mg_{zw})+2w^m g_z f_{zw},\\
I_2 &=w^m g_{z}f_{ww}-(1+ f_z)(wg_0''+2g_0'+m(m-1)w^{m-2}g+2mw^{m-1}g_w+w^mg_{ww})-\\
&-2f_w(mw^{m-1}g_z+w^mg_{zw}) +2(1+wg_0'+g_0+mw^{m-1}g+w^mg_w) f_{zw},\\
I_3 &=(1+wg_0'+g_0+mw^{m-1}g+w^mg_w)f_{ww}-\\
&-f_w(wg_0''+2g_0'+m(m-1)w^{m-2}g+2mw^{m-1}g_w+w^mg_{ww}).
\end{aligned}\eeq
Importantly, \eqref{trule3} is an identity in the {\em free variables $z,w,\zeta$}, where the latter triple runs a suitable open neighborhood of the origin in $\CC{3}$.

Though the expressions \eqref{trule3},\eqref{2jet} look cumbersome, we will shortly be able to work out \eqref{trule3} elegantly.

After that, let us equalize in \eqref{trule3} terms with a fixed degree $k,\,0\leq k\leq 3$ in $\zeta$. We expand
\begin{equation}\Label{expandPhi}\label{expandPhi}
\Phi(z,w,\zeta)=\sum_{k\geq 0} \Phi_k(z,w)\zeta^k,
\end{equation}
and similarly  $\Phi^*$, where we assume that $\Phi^*$ satisfies \eqref{expandPhi1}, and obtain:
\beq\begin{aligned}\Label{PDE-1}\label{PDE-1}
I_0 &=\Phi_0 + \ldots\\
I_1 &=\Phi_1 + \ldots\\
w^{m}I_2 &=\Phi_2-\Phi^*_2+\ldots\\
w^{2m}I_3 &=\Phi_3-\Phi^*_3+\ldots,
\end{aligned}\eeq
where $\ldots$ signify {\em convergent power series}, without constant terms, in the $1$-jet of $(f,g_0,g)$, whose coefficients only depend on the source defining function $\Phi^*$ (and are independent of the target $\Phi$!).
We have

\begin{proposition}\Label{Cproblem}\label{Cproblem}  Let $ \mathcal{E}^* \in \mathfrak{E}$, and
$f_0(w)$, $f_1(w)$, $g_0 (w)$, $g_1 (w)$ be power series satisfying $f_0 (0) = g_0 (0) =  g_1(0) = f_1 (0) = 0$.
Then there exist unique $\mathcal{E} \in \mathfrak{E}$ and
 $H= (\tilde f,\tilde g) \in \Hom (\mathcal{E}, \mathcal{E}^*)$ such that
\begin{equation}\Label{initdata}\label{initdata}
\begin{aligned}
f_0(w)=\tilde f(0,w),\quad f_1(w)=\tilde f_z(0,w) - 1, \\ w g_0(w) = \tilde g(0,w) - w,\quad w^m g_1(w)=\tilde g_z(0,w).
\end{aligned}
\end{equation}
\end{proposition}

\begin{proof}
Because we are looking for
a singular ODE with right hand side $\Phi$ for which
$\Phi_0 = \Phi_1 = 0$, we need to prove that the first two equations in \eqref{PDE-1}, with $\Phi_0 = \Phi_1 = 0$, have a solution $f,g$.
We are going to show that this is uniquely possible
 with the given initial data \eqref{initdata}.


In order to study the system of PDEs under discussion, let us use the first two formulas in \eqref{PDE-1} as a system of equations determining $\Phi_0 $
 and
 $\Phi_1$, respectively, which we reqire to be zero. Using the first two equations in  \eqref{2jet}, this yields system of
 PDEs, linear in the second
derivatives $f_{zz},g_{zz}$,  which
can be written as
\[ \begin{pmatrix}
	g_z & - (1 + f_z)\\
	(1 + w g_0' + g_0 + mw^{m-1} g + w^m g_w) & -w^m f_w
\end{pmatrix} \begin{pmatrix}
	f_{zz} \\ g_{zz}
\end{pmatrix} = \dots . \]
In view of \eqref{initdata}, the determinant of the matrix is nonzero at the point $(z,w)=(0,0)$. Hence, applying Cramer's rule, we obtain an {\em analytic} system of the kind
\begin{equation} \Label{system1}\label{system1}
\begin{aligned}
f_{zz}=U(z,w,g_0,g_0',f,g,f_z,g_z,f_w,g_w,f_{zw},g_{zw}), \\
g_{zz}=V(z,w,g_0,g_0',f,g,f_z,g_z,f_w,g_w,f_{zw},g_{zw})
\end{aligned}
\end{equation}
for some  germs of holomorphic functions  $U,V$ at the origin, which depend only on $\Phi^*$.
The Cauchy-Kowalevskaya theorem
guarantees
the existence of a (unique) solution to \eqref{system1} with initial  data
$$f(0,w)=f_0(w), \quad f_z(0,w)=f_1(w), \quad g(0,w)=0, \quad g_1(0,w)=g_1(w), $$
we determine a unique (resp. formal or holomorphic near the origin) solution $f,g$ for \eqref{system1}.

The associated functions $\tilde f (z,w) = z + f(z,w)$, $\tilde g (z,w) = w + w g_0 (w) + g(z,w)$ transform $\mathcal{E}^*$ to the (up to the
initial data unique) $\mathcal{E}$. The initial conditions also
imply that $(\tilde f, \tilde g)$ is
of the form required in \eqref{normalmap}.
\end{proof}

In view of \autoref{Cproblem}, the unique determination of a map $(\tilde f,\tilde g)$
is equivalent to determining the {\em Cauchy data for $H$}, i.e.
the vector
\begin{equation}\Label{Cdata}\label{Cdata}
Y(w):=\bigl(f_0(w),f_1(w),g_0(w),g_1(w)\bigr).
\end{equation}
 In order to uniquely determine the Cauchy data $Y(w)$,
 we have to  put further normalization conditions on the
 function $\Phi$.  Any such collection of
 normalization condition will be
 referred to as
  {\em a normal form} of an ODE \eqref{ODE}. We describe them below in three different cases,
  based on the decomposition
  \begin{equation}
  	\label{e:expandPhi2} \Phi(z,w,\zeta) = \sum_{a,b,c} \Phi_{a,b,c} z^a w^b \zeta^c.
  \end{equation}

For $m=1$, we define the space $\mathcal D_1$ to consist
of all formal power series $\Phi(z,w,\zeta)\in \zeta^2 \fps{z,w,\zeta}$ which in addition to satisfying $\Phi=O(\zeta^2)$ also satisfy
\begin{equation}\Label{nspaceODE1}\label{nspaceODE1}
\Phi_{0,j,2}=\Phi_{1,j,2}=\Phi_{0,j,3}=\Phi_{1,j,3}=0, \quad j>0,
\end{equation}
or equivalently,
\begin{equation}\Label{nspaceODE1a}\label{nspaceODE1a}
\dopt{^{j+k}\Phi}{z^j \zeta^k} (0,w,0) = \dopt{^{j+k}\Phi}{z^j \zeta^k} (0,w,0), \quad j = 0,1 , \quad k =2,3.
\end{equation}

For $m>1$, we need, for technical reasons that will become clear later, to first fix a real vector $$\sigma=(\sigma_1,\sigma_2)\in\RR{2}.$$
{For each such fixed $\sigma = (\sigma_1, \sigma_2)$},
we define the space $\mathcal D^\sigma_m$ to consist of all
formal power series $\Phi(z,w,\zeta)\in \zeta^2 \fps{z,w,\zeta}$
which in addition to satisfying $\Phi=O(\zeta^2)$ also satisfy\begin{equation}\Label{nspaceODEm}\label{nspaceODEm}
\begin{aligned}
\Phi_{0,j,2}&=
\begin{cases}
 0 & j>0,j\neq m-1 \\ m & j =m-1;
 \end{cases} \\
\Phi_{1,j,2}& =\Phi_{0,j,3}=0, \,\, j>0;\\
\Phi_{1,j,3}&=
\begin{cases}
0 & j>0,\,\,j\notin \{m-1,2m-2,3m-3\}\\
\sigma_1 & j = 2m-2 \\
\sigma_2 & j = 3m -3.
\end{cases}
\end{aligned}
\end{equation}

We will denote the corresponding subcategories of
equations whose right hand side is in
$\mathcal{D}_m^\sigma$ or $\mathcal{D}_1$, respectively, by
$\mathfrak{D}_m^\sigma \subset \mathfrak{E}_m$
and $\mathfrak{D}_1 \subset \mathfrak{E}_1$.


\subsection{A singular system of ODEs for the Cauchy data}
We will show that  under a generic condition on the defining function $\Phi^*$ of an ode $\mathcal{E}^*$,   which we will call the {\em non-resonancy} of the ODE
$\mathcal{E}^*$,
the condition
\[ \Hom (\mathcal{E}, \mathcal{E}^*) \neq \emptyset,
\quad
\mathcal{E}\in \mathfrak{D}_1 \text{ (or }
\mathfrak{D}_m^\sigma \text{, respectively)}\]
 determines the Cauchy data
 $f_0(w),f_1(w),g_0(w),h_1(w)$ (almost) uniquely.  However, the analytic regularity of the resulting functions for a system arising from
 an $m$-infinite type hypersurface as well as the non-resonancy condition itself depend significantly
 on $m$ and subsequently also (for $m>1$) on whether the ODE $\mathcal{E}$ is of {\em Fuchsian type} or it is not (this condition is introduced and discussed in Section 8 below). Accordingly, we consider  the normalization procedure in the latter three distinct cases, two of which are justified below and the last one considered in Section 8. The goal of this
 section is to collect the results which apply to all the systems, regardless
 of the specific $m$.

In the sequel, we consider the expansions
\begin{equation}\Label{ABC}\label{ABC}
\Phi_2(z,w)=\sum_{j\geq 0} A_j(w)z^j, \quad \Phi_3(z,w)=\sum_{j\geq 0} B_j(w)z^j, \quad \Phi_4(z,w)=\sum_{j\geq 0} C_j(w)z^j
\end{equation}
and
\begin{equation}\Label{expandmap}\label{expandmap}
f=f_0+f_1z+f_2z^2+...,\quad g=g_1z+g_2z^2+...
\end{equation}

In order to obtain uniqueness conditions for the collection \eqref{initdata}, we will make use of the last two equations in \eqref{PDE-1}. Some of the terms can
actually not be changed by transformations of the form we consider:

\begin{lemma}\Label{lowdegree}\label{lowdegree}
Let $\mathcal{E},\mathcal{E}^* \in \mathfrak{E}$ and
assume that $\Hom(\mathcal{E},\mathcal{E}^*) \neq \emptyset$.
Then
\begin{equation}\Label{lowPhi}\label{lowPhi}
\Phi_{002} = \Phi_{002}^* ,\,
\Phi_{102}=\Phi_{102}^*,\,
\Phi_{003}=\Phi_{003}^*,\,
\Phi_{103}=\Phi_{103}^*.
\end{equation}
\end{lemma}
The proof of this Lemma is
a straightforward computation of the corresponding terms in the transformation rule \eqref{trule3} by using the vanishing of the Cauchy data at $w=0$ and \eqref{PDE-1}.

\begin{observation}\Label{NFconditions}\label{NFconditions}
For $m=1$, the normal form condition $\Phi\in\mathcal D_1$ determines the four coefficient functions $A_0(w),A_1(w),B_0(w),B_1(w)$ completely, modulo their constant terms \eqref{lowPhi}.
Unlike that, if $m>1$ the normal form condition $\Phi\in\mathcal D^\sigma_m$ determines (for each fixed $\sigma$) the four coefficient functions $A_0(w),A_1(w),B_0(w),B_1(w)$ completely, modulo their constant terms and the coefficient
\begin{equation}\Label{unknown}\label{unknown}
\Phi_{1,m-1,3} = \frac{1}{(m-1)!}\dopt{^{m-1} B_1}{w^{m-1}} (0) .
\end{equation}
\end{observation}
We will discuss the actual freedom in a possible determination of this parameter later in Section 6. \begin{convention}\Label{star}\label{star}
We denote the undetermined parameter $\Phi_{1,m-1,3}$ as above
in what follows by $*$.
\end{convention}

The preliminary result for preparing the
ODEs for the
Cauchy data can be summarized in the following proposition. After this
preliminary result,  the classification procedure becomes different depending on $m$ and the defining function $\Phi^*$.

\begin{proposition}
\label{p:preparesystem} Let $\mathcal{E}^* \in \mathfrak{E}_m$. There exists a  germ
of an analytic map $T = (T_1, T_2, T_3, T_4)$ defined
near the origin in $\C^9$, vanishing at the origin,
and depending only on $\mathcal{E}^*$ (and in case $m>1$
also the undetermined coefficient $*$) such that
if
$(f,g) \in \Hom(\mathcal{E},\mathcal{E}^*)$, where
$\mathcal{E} \in \mathfrak{D}_1$ (or $\mathfrak{D}_m^\sigma$, respectively), then
\begin{equation}\Label{merom}\label{merom}
\begin{aligned}
w^{m+1}g_0''&=T_1(w,g_0,g_1,f_0,f_1,wg_0',w^{m}g_1',w^{m}f_0',w^{m}f_1'),\\
w^{2m}g_1''&=T_2(w,g_0,g_1,f_0,f_1,wg_0',w^{m}g_1',w^{m}f_0',w^{m}f_1'),\\
w^{2m}f_0''&=T_3(w,g_0,g_1,f_0,f_1,wg_0',w^{m}g_1',w^{m}f_0',w^{m}f_1'),\\
w^{2m}f_1''&=T_4(w,g_0,g_1,f_0,f_1,wg_0',w^{m}g_1',w^{m}f_0',w^{m}f_1').
\end{aligned}
\end{equation}
On the other hand, if $(f_0,f_1, g_0, g_1)$ satisfy \eqref{merom}, then
there exist a map $(f,g)$ and $\mathcal{E} \in \mathfrak{D}_1$ (or $\mathfrak{D}_m^\sigma$, respectively) such that $(f,g) \in \Hom(\mathcal{E},\mathcal{E}^*)$.
\end{proposition}

\begin{proof}
Our normalization conditions $\Phi\in\mathcal D_1$, and $\Phi\in\mathcal D_m^\sigma$ respectively,
translate the
equations coming from the corresponding
terms in \eqref{PDE-1}
into four ODEs in $H = (f_0, f_1 , g_0, g_1)$ by the following process.

\smallskip

(i) We set in the last two equations in \eqref{PDE-1} either
\begin{equation}\Label{Phi23}\label{Phi23} \left.
\begin{aligned}
&\Phi_{002}=\Phi^*_{002},\, \Phi_{102}=\Phi^*_{102},\,\Phi_{003}=\Phi^*_{003}, \,\Phi_{103}=\Phi^*_{103},\\
&\Phi_{0j2}=0=\Phi_{1j2}=\Phi_{0j3}=\Phi_{1j3}=0,\,\, j\geq 1.
\end{aligned}\right\rbrace \text{ if }m=1
\end{equation}
 or
\begin{equation}\Label{m-1}\label{m-1} \left.
\begin{aligned}
&\Phi_{002}=\Phi^*_{002},\, \Phi_{102}=\Phi^*_{102},\,\Phi_{003}=\Phi^*_{003}, \,\Phi_{103}=\Phi^*_{103},\\
& \Phi_{0j2} = \begin{cases}
0 & j\in \N\setminus \left\{ m-1 \right\}  \\
m & j = m
\end{cases} \\
&\Phi_{1j2}=\Phi_{0j3}=0,\,\, j\geq 1,\\
&\Phi_{1j3}= \begin{cases}
0 & j\in \N \setminus \{m-1,2m-2,3m-3\}  \\
\sigma_1 & j=2m -2 \\
\sigma_2 & j = 3m-3.
\end{cases}
\end{aligned}\right\rbrace\text{ if }m>1.
\end{equation}

\smallskip

(ii) After performing (i),  we collect
the terms of degrees $0$ and $1$ in $z$ in the last two equations in
 \eqref{PDE-1}. Note that, for $m>1$, the coefficient \eqref{unknown} remains undetermined and thus becomes an unknown parameter for the system of four ODEs under discussion.

By performing (i),(ii), we obtain four second order ODEs, depending on $f_2,g_2$, as well as
 $f_0(w),f_1(w),g_0(w),h_1(w)$ and their derivatives of order $\leq 2$ (because $f,g$ appear in the last two equations of \eqref{PDE-1}differentiated in $z$ at most once), and the ODEs depend only on $\Phi^*$ and the parameter $*$ (because
  the right hand side of \eqref{PDE-1} only depended on $\Phi^*$).

In order to eliminate $f_2,g_2$ from the four ODEs under discussion, we use the equations \eqref{system1}, evaluated at $z=0$, which allow us to
express  $f_2,g_2$ as {\em analytic} functions of  $f_0(w),f_1(w),g_0(w),g_1(w)$ and their derivatives of order $\leq 1$.
After substituting these two expressions into the four second order ODEs under discussion, we obtain a collection of four second order ODEs in $f_0(w),f_1(w),g_0(w),g_1(w)$ and their derivatives of order $\leq 2$ {\em only}.
The coefficient \eqref{m-1} (for $m>1$) remains {\em a parameter} in this system.
We denote the  system of ODEs by $\mathcal{F}$. Before we discuss the form of the
system of ODEs obtained by the procedure outlined above, let us note that we
have proved the statement about any $(f,g) \in \Hom(\mathcal{E}, \mathcal{E}^*)$
having the property that $(f_0, f_1, g_0, g_1)$  is a solution of $\mathcal{F}$.
On the other hand, if we have a formal or
analytic solution $(f_0, f_1, g_0, g_1)$ of $\mathcal{F}$,
then we can apply
 \autoref{Cproblem} to see that it uniquely defines a
 (formal or analytic) equation $\mathcal{E}\in \mathfrak{D}_1$ or
 $\mathcal{E}\in \mathfrak{D}_m^\sigma$ and a  (resp. formal or analytic) transformation $(f,g)\in \Hom(\mathcal{E}, \mathcal{E}^*)$.

Now, let us discuss how one obtains the exact form of the ODEs in \eqref{merom}. First
we note that
when solving the first two equations in \eqref{PDE-1} for $f_{zz}, g_{zz}$ in order to obtain \eqref{system1},  all the derivatives $f_w,g_w,f_{zw},g_{zw}$ appear with the factor $w^m$, and the derivative $g_0'$ appear with the factor $w$. In accordance with that, when we express $f_2,g_2$ as analytic functions of  $f_0(w),f_1(w),g_0(w),g_1(w)$ and their derivatives of order $\leq 1$ as discussed above,
 the derivatives $f_0',f_1',g_1'$ appear with the factor $w^m$, and the derivative $g_0'$ appears with the factor $w$.

Further, \eqref{2jet} implies that
the second order derivatives $g_0'',g_1'',f_0'',f_1''$ appear in $\mathcal F$ {\em linearly}.  Now, using    \eqref{trule3}, \eqref{PDE-1}, as well as our observation above for $f_2$ and $g_2$ we can see that $\mathcal S$ is a {\em linear} system in
$$w^{m+1}g_0'',w^{2m}g_1'',w^{2m}f_0'',w^{2m}f_1'',$$
in which the derivatives $f_0',f_1',g_1'$ appear with the factor $w^m$, and the derivative $g_0'$ appears with the factor $w$.

The matrix of this linear system  evaluated at the origin in the space of $1$-jets of maps $w\mapsto (g_0(w),g_1(w),f_0,f_1)$ equals
$$\mbox{diag}\{-1,-1,1,1\}.$$ Hence we can solve the linear system under discussion in $w^{m+1}g_0'',w^{2m}g_1'',w^{2m}f_0'',w^{2m}f_1''$ and obtain the system of singular holomorphic ODEs claimed in \eqref{merom}. Note that if the system has a solution,
then the $T$'s have no constant term, and so we can just
assume that the right hand side $T$ vanishes at $0$.
 \end{proof}

\section{Case $m=1$}
In order to define our non-resonancy condition, we start by
 constructing a first-order system from
\eqref{merom}. Writing
$$H:=(g_0,g_1,f_0,f_1),$$
the system \eqref{merom}, in case $m=1$, takes the following form:
\begin{equation}\Label{merom1}\label{merom1}
w^{2}H''=T(w,H,wH')
\end{equation}
for a holomorphic map
 $T$ defined near and vanishing at the origin.  We can
 rewrite \eqref{merom1} as a first order
 ODE  using the substitution
\begin{equation}\Label{subst}\label{subst}
G:=wH', \quad \mathbf{ H}:=(H,G).
\end{equation}
We write the resulting first order ODE as
\begin{equation}\Label{merom2}\label{merom2}
w\mathbf{ H}'=\mathbf{ T}(w,\mathbf{ H})
\end{equation}
where $\mathbf{ T}$ is a holomorphic
map defined near the origin which satisfies
$\mathbf{ T}(0,0)=0$. Note that $\mathbf{ H}(0)=0$. The singular holomorphic ODE \eqref{merom2} is within the class of {\em Briot-Bouquet ODEs} (see Section 2), hence every formal solution of
\eqref{merom2} which satisfies
the condition $\mathbf{ H}(0)=0$) is convergent.
For the existence of formal (and hence holomorphic) solutions, it is known that such a solution with $\mathbf{ H}(0)=0$ exists and is  unique if
 the linearization matrix
\begin{equation}\Label{mat}\label{mat}
\mathbf{M}:=\left.\frac{\partial \bold T}{\partial \bold H}\right|_{w=\bold H=0}
\end{equation}
has no positive integer eigenvalues, or equivalently, if  the
intimately related {\em Euler system}
\begin{equation}\Label{Euler}\label{Euler}
w\bold H'=M \mathbf{H}
\end{equation}
has only the trivial solution.

Note that the matrix $\mathbf{M}$ defined
by  \eqref{mat}, is actually given by
\[ \mathbf{M} = \begin{pmatrix}
	0 & I \\ T_H & I + T_G
\end{pmatrix} \biggr|_{w = H = G = 0}, \]
with
$T$ defined in  \autoref{p:preparesystem}. T, and therefore $\mathbf{M}$ only depend
on  $\mathcal E^*$. We can therefore make the following definition.

\begin{definition}\Label{nonres}\label{nonres}
An ODE $\mathcal E \in \mathfrak{E}_1$
 is called {\em non-resonant}
 if the matrix given in  \eqref{mat} does not have positive integer eigenvalues. A real hypersurface \eqref{madmissiblereal} is called {\em non-resonant}, if its associated ODE $\mathcal{E} (M)$ (see \eqref{ODE}) is non-resonant.
\end{definition}

Consider now our  system \eqref{merom1}. As discussed above, under the non-resonancy condition, it  has a unique holomorphic solution with $H(0)=0$. Therefore,
\autoref{p:preparesystem} implies the following.

\begin{theorem}\Label{nfODE}\label{nfODE}
If $\mathcal{E} \in \mathfrak{E}_1$ is nonresonant, then
there exists a unique normal form $\mathcal{N} \in \mathfrak{D}_1$ such that
$\Hom(\mathcal{N}, \mathcal{E}) \neq \emptyset$. Indeed, in that
case, there is a unique normalizing transformation:
$|\Hom(\mathcal{N}, \mathcal{E}) | =1$.  If $\mathcal{E}$ is
convergent, so is $\mathcal{N}$ and the normalizing transformation.
\end{theorem}

For the rest of this section, we will discuss how
 to write down the non-resonancy condition explicitly,
 using the linearization of the system \eqref{merom2} in
$\bold H$, i.e. the linear system
\begin{equation}\Label{lnd}\label{lnd}
w\mathbf{H}'=\mathbf{L}(w) \mathbf{H}, \quad \text{where } \mathbf{L}(w):=\left.\frac{\partial \mathbf{T}}{\partial \mathbf{H}}\right|_{\mathbf{H}=0}.
\end{equation}
This linearized system
can be obtained by entirely repeating the above computational scheme of
identifying a map \eqref{normalmap} under consideration, but at the same time {\em ignoring all the non-linear terms in the variables $f,g,g_0$ within the basic identity \eqref{trule3}}. This latter procedure precisely corresponds to searching for {\em infinitesimal automorphisms} (Lie symmetries)
\begin{equation}\Label{inf}\label{inf}
L=P(z,w)\dz+Q(z,w)\dw
\end{equation}
of the initial ODE $\mathcal E^*$ with
\begin{equation}\Label{PQ}\label{PQ}
\begin{aligned}
P(z,w)&=f_0(w)+f_1(w)z+f_2(w)z^2+\cdots, \\
Q(z,w)&=wg_0(w)+w^mg_1(w)z+w^mg_2(w)z^2+\cdots
\end{aligned}
\end{equation}
(Note that, according to \autoref{specialmap}, any infinitesimal automorphism $L$ of a hypersurface \eqref{madmissiblereal} has the form \eqref{PQ}).

Second, we observe that the linearization matrix $\mathbf{M}$ in \eqref{mat}
is given by $$\mathbf{M}=\mathbf{L}(0).$$ In view of that, we can explicitly determine the matrix, and therefore
 the non-resonancy condition, in the following way.

\smallskip

\noindent (i) we write down the condition that a vector field $L$ satisfying \eqref{PQ} is an infinitesimal automosphism (Lie symmetry) of the ODE $\mathcal E^*$.

\smallskip

\noindent (ii) we extract from the respective basic identity terms with
$$z^kw^l(w')^0, \,  z^kw^l(w')^1, \,  z^0w^l(w')^2, \,  z^1w^l(w')^2, \,  z^0w^l(w')^3, \,  z^1w^l(w')^3, \quad k,l\geq 0$$
and obtain a system of linear PDEs for $P,Q$, which gives us a system of four second order linear ODEs for the initial components $g_0,g_1,f_0,f_1$  of $L$. A substitution identical to \eqref{subst} into the latter system gives us the first order system of eight ODEs \eqref{merom2}.

\smallskip

\noindent (iii) finally, we investigate formal solutions of the intimately related Euler system \eqref{Euler} (obtained by putting $w=$ into the right hand side), and write down the non-resonancy condition. In fact, the latter can be done either for the first order order system of 8 ODEs or for the second order system of four ODEs.

\smallskip

To perform step (i), we proceed similarly to \cite{analytic} and search for a Lie symmetry $L$ of an ODE \eqref{ODE} (satisfying \eqref{PQ}) via the jet prolongation method (see Section 2). That is, we consider the ODE $\mathcal E^*$ as a submanifold in the space $J^2(\CC{},\CC{})$ of $2$-jets of functions $\CC{}\mapsto\CC{}$ with the coordinates $z,w,w_1,w_2$. After that, we write down the fact that the {\em second jet prolongation}
\begin{equation}\label{prolong}
L^{(2)}=P(z,w)\dz+Q(z,w)\dw+Q^{(1)}(z,w,w_1)\frac{\partial}{\partial w_1}+Q^{(2)}(z,w,w_1,w_2)\frac{\partial}{\partial w_2},
\end{equation}
of a Lie symmetry $L$
is tangent to this submanifold. We have
\begin{gather*}w_1:=w\rq{},\quad w_2:=w\rq{}\rq{},\\
 Q^{(1)}=Q_z+\bigl(Q_w-P_z\bigr)w_1-P_w(w_1)^2,\\
Q^{(2)}=Q_{zz}+\bigl(2Q_{zw}-P_{zz}\bigr)w_1+\bigl(Q_{ww}-2P_{zw}\bigr)(w_1)^2-\\
-P_{ww}(w_1)^3
+\bigl(Q_w-2P_z)w_2-3P_w w_1w_2,
\end{gather*}
and the tangency condition means
\begin{equation}\label{tangency}
Q^{(2)}\left|_{w_2=\Phi\bigl(z,w,\frac{w_1}{w^m}\bigr)}\right.=\Phi_z\left(z,w,\frac{w_1}{w^m}\right)\cdot P + \Phi_w\left(z,w,\frac{w_1}{w^m}\right)\cdot Q+\frac{1}{w^m}\Phi_{\zeta}\left(z,w,\frac{w_1}{w^m}\right)\cdot Q^{(1)}
\end{equation}
for all $(z,w,w_1)$ lying in an appropriate open set $V\subset J^1(\CC{},\CC{})$.

To perform step (ii), we first gather in \eqref{tangency}  terms with $(w_1)^0, (w_1)^1, (w_1)^2, (w_1)^3$.  Using the notations
$$a:=\frac{1}{w^{m}}\Phi_2,\quad b:=\frac{1}{w^{2m}}\Phi_3,\quad c:=\frac{1}{w^{3m}}\Phi_4,$$
we get, respectively:
\begin{equation}\label{initial}
\begin{aligned}
Q_{zz} &=0,\\
 2Q_{zw}-P_{zz}&=2aQ_z,\\
 Q_{ww}-2P_{zw}&=a(-Q_w+2P_z)+a_zP+a_wQ+3bQ_z+2a(Q_w-P_z),\\
 P_{ww}&=b(Q_w-2P_z)-aP_w-b_zP-b_wQ-4cQ_z+3b(P_z-Q_w).
\end{aligned}
\end{equation}
By employing the expansion \eqref{PQ}, from the first equation in \eqref{initial} we get
 \begin{equation}\label{findQ}
  Q(z,w)=wg_0(w)+w^mg_1(w)z,
  \end{equation}
and from the second
\begin{equation}\label{findP}
P(z,w)=f_0(w)+f_1(w)z+(w^mg_1\rq{}(w)+mw^{m-1}g_1)z^2-2w^m\tilde a(z,w)g_1(w),
\end{equation}
where $\tilde a_{zz}=a, \, \tilde a=O(z^2)$. For $m=1$ this gives, in particular,
\begin{equation}\Label{f2g2}\label{f2g2}
g_2=0, \quad f_2=-A_0g_1+wg_1'+g_1
\end{equation}
(in the notations of \eqref{ABC}).
After that, we gather in the last two equations in \eqref{initial} terms with
$$z^0w^l(w')^2, \,  z^1w^l(w')^2, \,  z^0w^l(w')^3, \,  z^1w^l(w')^3,$$ respectively, and substitute \eqref{f2g2} into the result. This gives us a system of four second order {\em Fuchsian} linear ODEs. It is not difficult to compute that this system looks as follows (we use the notations in \eqref{ABC}):
\begin{equation}\Label{FS}\label{FS}
\begin{aligned}
w^2g_0''+w(2-A_0)g_0'-2wf_1'-A_0'wg_0-3B_0g_1-A_1f_0&=0\\
3w^2g_1''+A_1wg_0'+3wA_0g_1'+wA_1'g_0+(3B_1-3wA_0')g_1+2A_2f_0+A_1f_1&=0\\
w^2f_0''+2B_0wg_0'+A_0wf_0'+wB_0'g_0+4C_0g_1+B_1f_0-B_0f_1&=0\\
w^2f_1''+2B_1wg_0'+A_1wf_0'+A_0wf_1'+wB_1'g_0+(2B_0(A_0-1)+wB_0'+4C_1)g_1+2B_2f_0&=0.
\end{aligned}
\end{equation}
For step (iii), we have to evaluate the coefficients of the system \eqref{FS} at $w=0$, and search for a formal solution $H(w)$ of the arising system with constant coefficients. If we expand a formal solution $H(w)$ of it as
$$H(w)=\sum_{k\geq 1} h_kw^k$$
and the coefficients of \eqref{FS} as
\begin{equation}\Label{expandABC}\label{expandABC}
\alpha_j:=A_j(0),\quad \beta_j:=B_j(0), \quad \gamma_j:=C_j(0),\quad j=0,1,2.
\end{equation}
Then we obtain for the Taylor coefficients $h_k,\,k\geq 1$ the following set of equations:
\begin{equation}\Label{hhk}\label{hhk}
\begin{pmatrix}
k(k+1-\alpha_0) & -3\beta_0 & -\alpha_1 & -2k\\
\alpha_1k & 3k(k+1-\alpha_0)+3\beta_1 & 2\alpha_2 & \alpha_1\\
2k\beta_0 & 4\gamma_0 & k(k-1+\alpha_0)+\beta_1 & -\beta_0\\
2k\beta_1 & 2\beta_0(\alpha_0-1)+4\gamma_1 & k\alpha_1+2\beta_2 & k(k-1+\alpha_0)
\end{pmatrix}
\cdot h_k=0.
\end{equation}
The non-resonancy condition is satisfied if and only if each of the equations \eqref{hhk} for $k\geq 1$ has only the trivial solution, that is, if the equation
\begin{equation}\Label{hk}\label{hk}
\mbox{det}\,
\begin{pmatrix}
(k+1-\alpha_0) & -3\beta_0 & -\alpha_1 & -2k\\
\alpha_1 & 3k(k-1+\alpha_0)+3\beta_1 & 2\alpha_2 & \alpha_1\\
2\beta_0 & 4\gamma_0 & k(k-1+\alpha_0)+\beta_1 & -\beta_0\\
2\beta_1 & 2\beta_0(\alpha_0-1)+4\gamma_1 & k\alpha_1+2\beta_2 & k(k-1+\alpha_0)
\end{pmatrix}
=0
\end{equation}
has no solutions $k\in\mathbb N$.

We summarize this result as following.
\begin{proposition}\Label{nonres2}\label{nonres2}
An ODE \eqref{ODE} is non-resonant if and only if the algebraic equation \eqref{hk} of degree $7$ (where $\alpha_j,\beta_j,\gamma_j$ are taken from \eqref{expandABC}) has no positive integer solutions.
\end{proposition}
In particular, a hypersurface
 will have at most finitely many resonances (i.e. integer solutions to the equation \eqref{hk}).
\begin{corollary}\Label{finiteres}\label{finiteres}
There can exist only finitely many resonances for a hypersurface \eqref{madmissiblereal} when $m=1$ (in fact, at most $7$ of them).
\end{corollary}
Using the explicit characterization of ODE \eqref{ODE} associated with {\em generically spherical} hypersurfaces \eqref{madmissiblereal}, it is not difficult to prove also
\begin{proposition}\Label{nontrivial}\label{nontrivial}
There exist real hypersurfaces \eqref{madmissiblereal} which are non-resonant at $0$. Accordingly, a generic hypersurface \eqref{madmissiblereal} is non-resonant at $0$  (in the sense of the jet topology in the space of defining functions $h(z,\bar z,u)$).
\end{proposition}
\begin{proof}
Fix $k\geq 1$ and consider the determinant in \eqref{hk} as a linear expression in $\alpha_2,\gamma_0$, and the product $\alpha_2\gamma_0$.  Then the coefficient of $\alpha_2\gamma_0$ equals, up to a nonzero factor, to
\begin{equation}\Label{a2c0}\label{a2c0}
(k-1+\alpha_0)(k+1-\alpha_0)+4\beta_1.
\end{equation}
Fix  a generic pair $\alpha_0,\beta_1$ such that the expression \eqref{a2c0} does not vanish for all natural $k$ (by using formulas \eqref{relations},\eqref{phitoh} below, this is accomplished by an appropriate choice of the Taylor coefficients $h_{22}(0),h_{33}(0)$ of the real defining function). On the other hand, one can see again from \eqref{relations},\eqref{phitoh} that, by varying $h_{24}(0)$, the coefficients $\alpha_2,\gamma_0$ are (up to the constant factor $6$) arbitrary complex conjugated complex numbers. Hence, if we fix $h_{23}(0),h_{34}(0)$ in an arbitrary way, then for each fixed natural $k$ one can avoid vanishing of the determinant in \eqref{hk} by erasing (at most) a real curve from the complex plane parameterized by $h_{24}(0)$. This proves that there is a choice of $\alpha_2,\gamma_0$ making the determinant in \eqref{hk} nonvanishing for all natural $k$ at once, as required.
\end{proof}

\section{Case $m>1$}
In this section, we investigate {\em formal power series} solutions of the system \eqref{merom} in the case $m>1$.
In order to do so, we make a power series ansatz, and
expand a candidate for a formal vector solution
$H:=(g_0,g_1,f_0,f_1)$ as
$$H=\sum_{k\geq 1} h_kw^k$$
After we plug this ansatz into \eqref{merom},
the coefficients of  $w^k,\,k\geq 1$ in the resulting equation
give rise to  equations $E_k$,
and we will study try to  solve $E_k$ for $h_k$. That is, we aim to determine uniquely $h_1$ from $E_1$, $h_2$ from $E_2$, etc.

It is not difficult to see that

\smallskip

\noindent (i) each  $E_k$ only involves  $h_1,h_2,...,h_k$;

\smallskip

\noindent (ii) becaise $m>1$,   the left hand side of the equation $E_k$ does
not involve  $h_k$, while on the right hand side  only  {\em linear terms} in $g_0,g_1,f_0,f_1,wg_0'$ cam give rise to $h_k$.
 Moreover, in the respective linear expressions
$$\lambda_1(w) g_0+\lambda_2(w) g_1+\lambda_3(w) f_0+\lambda_4(w) f_1+\lambda_5(w) g_0'w$$
one has to evaluate the coefficients $\lambda_1(w),...,\lambda_5(w)$ at $w=0$.

\smallskip

In view of the above, each of the equations $E_k$ has the form:
\begin{equation}\Label{Ek}\label{Ek}
M_kh_k=R_k(h_1,...,k_{k-1}),
\end{equation}
where $M_k$ is a $4\times 4$ matrix, and $R_k$ is a polynomial in its variables (its concrete form depends on $\Phi^*$ and is of no interest to us here).

We will now
determine the form of the matrices $M_k$, and formulate the arising non-resonancy condition. For doing so, in view of the observations (i),(ii) above, we may use a procedure similar to that in the case $m=1$. We, however, should ignore in the arising system of second order linear differential equations for $g_0,g_1,f_0,f_1$ analogous to \eqref{FS} all the second order derivatives, as well as the derivatives $g_1',f_0',f_1'$. Now the system of equations analogous to \eqref{FS}  becomes:
\begin{equation}\Label{FS1}\label{FS1}
\begin{aligned}
&(2w^m-wA_0)g_0'+((m-1)A_0-A_0'w)g_0-3B_0g_1-A_1f_0=mw^{m-1}\\
&A_1wg_0'+[A_1(1-m)+wA_1']g_0+[(3B_1-3w^mA_0'+3m(m-1)w^{2m-2}]g_1+2A_2f_0+A_1f_1=0\\
&2B_0wg_0'+[(2-2m)B_0+wB_0']g_0+4C_0g_1+B_1f_0-B_0f_1=0\\
&2B_1wg_0'+[(2-2m)B_1+wB_1']g_0+(2B_0A_0+w^mB_0'-2mw^{m-1}B_0+4C_1)g_1+2B_2f_0=\\
&=* w^{m-1}+\sigma_1w^{2m-2}+\sigma_2w^{3m-3}
\end{aligned}
\end{equation}
(we recall that we are using \autoref{star}).
To investigate the formal solvability of \eqref{FS1}, we (i) evaluate the coefficients of the system \eqref{FS1} at $w=0$; (ii) substitute a formal power series $H(w)$ into the resulting system of ODEs with constant coefficients; (iii) put zero for the right hand side (that is, switching to  the respective homogeneous system). In the notation of \eqref{expandABC} this gives
\begin{equation}\Label{hk1}\label{hk1}
\begin{pmatrix}
\alpha_0(k+1-m) & 3\beta_0 & \alpha_1 & 0\\
\alpha_1(k+1-m) & 3\beta_1 & 2\alpha_2 & \alpha_1\\
2\beta_0(k+1-m) & 4\gamma_0 & \beta_1 & -\beta_0\\
2\beta_1(k+1-m) & 2\beta_0\alpha_0+4\gamma_1 & 2\beta_2 & 0
\end{pmatrix}
\cdot h_k=0.
\end{equation}
for $k\neq m-1$, and
\begin{equation}\Label{hm-1}\label{hm-1}
\begin{pmatrix}
0 & 3\beta_0 & \alpha_1 & 0\\
0 & 3\beta_1 & 2\alpha_2 & \alpha_1\\
0 & 4\gamma_0 & \beta_1 & -\beta_0\\
0 & 2\beta_0\alpha_0+4\gamma_1 & 2\beta_2 & 0
\end{pmatrix}
\cdot h_k=0
\end{equation}
for $k=m-1$.
From \eqref{hk1}, we can see that, unlike the case $m=1$, {\em for the special value $k=m-1$ the respective matrix $M_{m-1}$ is always degenerate}. That is why, for {\em any} given source ODE $\mathcal E^*$, there is no hope to guarantee the uniqueness of a formal normalizing transformation \eqref{normalmap}. We then proceed as follows.

The special form of the matrices $M_k$ arising from \eqref{hk1},\eqref{hm-1} leads to
\begin{definition}\Label{nonres1}\label{nonres1}
An ODE \eqref{ODE} with $m>1$ is called {\em nonresonant}, if both the matrix
\begin{equation}\Label{44}\label{44}
M:=\begin{pmatrix}
\alpha_0 & 3\beta_0 & \alpha_1 & 0\\
\alpha_1 & 3\beta_1 & 2\alpha_2 & \alpha_1\\
2\beta_0 & 4\gamma_0 & \beta_1 & -\beta_0\\
2\beta_1 & 2\beta_0\alpha_0+4\gamma_1 & 2\beta_2 & 0
\end{pmatrix}
\end{equation}
and its upper right $3\times 3$ minor
\begin{equation}\Label{33}\label{33}
\tilde M:=\begin{pmatrix}
 3\beta_0 & \alpha_1 & 0\\
  3\beta_1 & 2\alpha_2 & \alpha_1\\
 4\gamma_0 & \beta_1 & -\beta_0\\
\end{pmatrix}
\end{equation}
are invertible.
A real hypersurface $M$ is called {\em non-resonant}, if its associated ODE is non-resonant.
\end{definition}
Let us again check that non-resonant hypersurfaces with $m>1$ actually exist.
\begin{proposition}\Label{nontrivialm}\label{nontrivialm} Let $m>1$.
There exists a real hypersurface $M\in \mathfrak{R}_m^\pm$
which is non-resonant at $0$. Accordingly, a generic hypersurface in $\mathfrak{R}_m^\pm$
(in the sense of the jet topology in the space of defining functions $h(z,\bar z,u)$) is non-resonant at $0$.
\end{proposition}
\begin{proof}
Let us first choose hypersurfaces $M$ for which the coefficient $\alpha_1$ (in the notations \eqref{expandABC}) is non-zero. To achieve this condition, we apply the formulas \eqref{relations} below and see that the condition under discussion is achieved by varying $\varphi_{32}(0)$, which can be chosen arbitrary complex thanks to the formulas \eqref{phih} below.  Then, for a generic $\gamma_0$ (and fixed $\beta_0,\beta_1$) the determinant of \eqref{33} is nonzero. However, the coefficient $\gamma_0$ can be as well chosen arbitrary complex, since it is proportional to $\varphi_{24}(0)$ (see \eqref{relations}), and $\varphi_{24}(0)$ is proportional to $h_{24}(0)$ (the proof of the latter fact is identical to the proof of \eqref{phih}, and we do not provide details). This finally proves that \eqref{33} can be achieved by choosing a generic collection $h_{23}(0), h_{33}, h_{24}(0)$.

We further consider the determinant of \eqref{44} as a linear expression in $\gamma_1$. The coefficient of $\gamma_1$ then equals up to a constant nonzero factor:
\begin{equation}\Label{coe}\label{coe}
3\alpha_1^2\beta_0-\alpha_0(\alpha_1\beta_1+2\alpha_2\beta_0).
\end{equation}
We can still vary the above triple $h_{23}(0), h_{33}, h_{24}(0)$ in order to keep the  coefficient of $\alpha_0$ nonzero in \eqref{coe}. As can be seen from \eqref{relations}, possible values of $\alpha_0$ run a (real) line in complex plane, hence we can finally choose $\alpha_0$ in such a way that \eqref{coe} is nonzero, and this shows that an appropriate choice of $\gamma_1$ (which is proportional to $h_{25}$) allows to finally make the determinant of \eqref{44} nonzero simultaneously with that of \eqref{33}, as required.
\end{proof}
 In the non-resonant case, we fix the parameter
\begin{equation}\Label{tau}\label{tau}
\tau:=\frac{1}{(m-1)!}g_0^{(m-1)}(0)=\frac{1}{m!}\frac{\partial^{m} G}{\partial w^{m}}(0,0)
\end{equation}
(it corresponds to the first component of the vector $h_{m-1}$, and as well corresponds to the parameter \eqref{realparam}). This parameter is the only one we can not determine uniquely, and its choice is with one-to-one correspondence with the choice of the coefficient \eqref{m-1}.
After that, in view of the non-resonancy condition, for each fixed $\tau$ we can uniquely solve the sequence of linear systems \eqref{hk1} with $k\neq m-1$, while for $k=m-1$ we can solve it, modulo the last component of the right hand side (this means that this first component of the right hand side of the respective linear system becomes a precise value). This gives us a unique formal solution for the set of equations \eqref{Ek}. We furthermore observe that
\begin{remark}\Label{invarcoef}\label{invarcoef}
The coefficient \eqref{unknown} is in fact independent of $\tau$, and thus becomes {\em an invariant} of an ODE \eqref{ODE}. This follows from the fact that the coefficient $\tau$ for the first time appears in the $m$-th equation in the above formal procedure, as earliest.
\end{remark}
By the above argument, we have proved
\begin{theorem}\Label{nfODE1}\label{nfODE1}
If $\mathcal{E} \in \mathfrak{E}_m$, where $m>1$, is non-resonant, then
for every $\tau\in \C$
there exists a normal form $\mathcal{N}_\tau \in \mathfrak{D}_m^\sigma$ and a formal transformation
$(f_\tau,g_\tau) \in \Hom (\mathcal{N}_\tau, \mathcal{E})$.
The normalizing transformation and the normal form are  unique (up to the complex parameter $\tau$  in \eqref{tau}).
\end{theorem}

\section{Solution of the equivalence problem}

In this section, we prove \autoref{theor1} and \autoref{theor2}.
In what follows, if not otherwise stated, we treat the case of a
positive real hypersurface $M \in \mathfrak{R}_m^+$
(i.e. $\epsilon=1$ in \eqref{madmissiblereal}),
as in the negative case the considerations are very analogous.

We start with transferring the normal form condition
$\mathcal{E} \in \mathfrak{D}^\sigma$
for ODE \eqref{ODE} to the associated $m$-admissible Segre families.
\begin{definition}\Label{segreinnf}\label{segreinnf}
We say that an $m$-admissible Segre family $\mathcal{S} \in \mathfrak{S}_m^+$
 {\em is in normal form},
 if its associated ODE  is in normal form, i.e. if
 $\mathcal{E} (\mathcal{S}) \in \mathfrak{D}_m^+$. We
write $\mathfrak{N}_m^{+,\sigma} = \mathcal{E}^{-1} (\mathfrak{D}_m^\sigma)$
for the space of
normal forms.
\end{definition}

We next formulate
\begin{proposition}\Label{relation}\label{relation}
The following relations hold between the coefficient functions \eqref{ABC} of an ODE \eqref{ODE} and its associated Segre family \eqref{mc}:
\begin{equation}\Label{relations}\label{relations}
\begin{aligned}
A_0(w)&=w^{m-1}\mp 2i\varphi_{22}(w), \quad A_1(w)=\mp 6i \varphi_{32} (w),\quad B_0(w)=- 2\varphi_{23}(w),\\
C_0(w)&=\pm 2i\varphi_{24}(w),\quad A_2(w)=\mp 12i\varphi_{42}(w), \\
B_1(w)&= - 6\varphi_{33}(w) \mp 2i(m-1)\varphi_{22}(w)w^{m-1}+8(\varphi_{22}(w))^2 \pm
2i\varphi'_{22}(w)w^m\\
B_2(w)&=-12\varphi_{43}(w)+36\varphi_{22}(w)\varphi_{32}(w)\pm 6i(1-m)w^{m-1}\varphi_{32}(w)\pm 6iw^m\varphi_{32}\rq{}(w)\\
C_1(w)&=\pm 6i\varphi_{34}(w)\mp 20i\varphi_{22}(w)\varphi_{23}(w)
+(4-4m)w^{m-1}\varphi_{23}(w)+2w^m\varphi_{23}\rq{}(w).
\end{aligned}
\end{equation}
\end{proposition}

Since the proof of Proposition~\ref{relation} is obtained by a direct substitution of the defining equation \eqref{mc} into an ODE \eqref{ODE}, and is as well very analogous to the proof of Proposition 4.2 in \cite{nonanalytic}, we leave the details to the reader.

Now \autoref{relation} implies
\begin{corollary}\Label{nftransfer}\label{nftransfer}
A Segre family $\mathcal{S} \in \mathfrak{S}_1^\pm$,
defined by \eqref{mc}, is in normal form
($\mathcal{S} \in \mathfrak{N}_1^\pm$) if and only if it satisfies:
\begin{equation}\Label{phi1a}\label{phi1a}
\varphi_{22}(w)=\varphi_{22}(0),\quad \varphi_{23}(w)=\varphi_{23}(0), \quad \varphi_{32}(w)=\varphi_{32}(0),  \quad \varphi_{33}(w)=\varphi_{33}(0),
\end{equation}
A Segre family $\mathcal{S} \in \mathfrak{S}_m^\pm$, where $m>1$,
, is in normal form
($\mathcal{S} \in \mathfrak{N}_m^{\sigma, \pm}$) if and only if it satisfies
\begin{equation}\Label{phi2m-2}\label{phi2m-2}
\begin{aligned}
&\varphi_{23}(w)=\varphi_{23}(0), \quad \varphi_{32}(w)=\varphi_{32}(0),
 \quad \varphi_{22}(w)=\lambda\pm\frac{i}{2}(m-1) w^{m-1}, \\
&\varphi_{33}(w)=\mu+\nu w^{m-1}+\sigma_1w^{m-1}+
 \sigma_2w^{3m-3}+\\
& +\frac{1}{3}\left(
\mp i(m-1)\varphi_{22}(w)w^{m-1}+4
(\varphi_{22}(w))^2\pm i\varphi'_{22}(w)w^m  \right),
\end{aligned}
\end{equation}
where $\lambda,\mu,\nu\in\CC{},\,\,\sigma_1,\sigma_2\in\RR{}$.
\end{corollary}



We now need a development of the technique for investigating the reality condition for parameterized families of curves introduced in \cite{divergence}.

 \begin{definition} We say that an $m$-admissible Segre
 family $\mathcal{S} \in \mathfrak{S}_m^\pm$
   has
 a {\em real structure} if it is the Segre family of an
 $m$-admissible real hypersurface $M\subset \CC{2}$, i.e. if
 $\mathcal{S} = \mathcal{S} (M)$ for some $M\in \mathfrak{R}_m^\pm$.
 We also say
 that an ODE $\mathcal{E}\in \mathfrak{E}_m^\pm$ has {\em $m$-positive
 (respectively, $m$-negative) real structure},
 if the associated positive (respectively, negative) $m$-admissible Segre family $\mathcal S^\pm_m(\mathcal E)$
 has a real structure. We say that the corresponding real hypersurface $M$ is {\em  associated } with $\mathcal E$.
 \end{definition}

  Let $\rho(z,y,u)$ be a holomorphic function near the origin in $\CC{3}$ with
  $\rho(0,0,0)=0$, and $d\rho(0,0,0)=du$. For  $z,\xi\in\Delta_\delta,w,\eta\in\Delta_\varepsilon$, let
 $$
 \mathcal S=\{w=\rho(z,\bar \xi,\bar \eta)\}
 $$
be a 2-parameter antiholomorphic family of holomorphic curves near
the origin, parametrized by $(\xi,\eta)$. An {\em admissible} (re)-parametrization
of  $\mathcal{S}$
is given by a function $\tilde \rho (z, \bar{\xi'}, \bar{\eta'})$ such that
\[ \mathcal{S} = \{ w = \tilde\rho(z,\bar{\xi'}, \bar{\eta'} ) \}  \]
and there exists a  germ of a biholomorphism $(\xi,\eta)\mapsto (\xi',\eta')$
such that $\rho(z,\bar \xi,\bar \eta) = \tilde \rho \left(z,\overline{\xi'(\xi,\eta)},\overline{\eta'(\xi,\eta)} \right)$. Fixing a parametrization and considering all  admissible (re)-parametrizations
gives rise to the notion of   a \emph{general Segre family}.

 For each point
 $p=(\xi,\eta)\in\Delta_\delta\times\Delta_\varepsilon$ we call
 the corresponding holomorphic curve
 $Q_p^\rho =\{w=\rho(z,\bar \xi,\bar \eta)\}\in \mathcal S$
its \it  Segre variety. \rm Clearly, an $m$-admissible Segre
family is a particular example of a general Segre family. Note that the
Segre varieties of a general Segre family do depend on the parametrization,
but admissible parametrizations give rise to a (holomorphic) relabeling of the Segre varieties.

We say that two general Segre families $\mathcal{S} $ and
 $\tilde {\mathcal{ S}}$ are equivalent if there exists
a germ of a biholomorphism $H=(f,g)$ of $(\CC{2},0)$ such that
$ \tilde {\mathcal{S}} = H^{-1} (\mathcal{S}) $, and such that
the solution of the implicit function problem
$g(z,w) = \tilde \rho( f(z,w)), \xi,\eta)$ for $w$ is an admissible
parametrization of $\mathcal{S}$.


Further, given a (general) Segre family
$\mathcal S$, from the implicit function theorem one concludes
that the antiholomorphic family of planar holomorphic curves
$$\mathcal
 S^{*,\rho}=\{\bar\eta=\rho(\bar\xi,z,w)\}$$ is also a general Segre family for
 some, possibly, smaller polydisc $\Delta_{\tilde\delta} \times \Delta_{\tilde\varepsilon}$, which depends on the chosen parametrization $\rho$.
 We note that for every admissible parametrization of $\mathcal{S}$, we obtain
 an equivalent Segre family.

 \begin{definition}
 The Segre family $\mathcal S^{*,\rho}$ is called the
 \emph{ dual Segre family } for $\mathcal S$ with the parametrization $\rho$.
 \end{definition}

 The dual Segre family has a simple interpretation:
 in the defining equation of the family $S$ one should consider the
 parameters $\bar\xi,\bar\eta$ as new coordinates, and the variables $z,w$ as new parameters.
 If we denote the Segre variety of a point $p$ with respect to the family
  $\mathcal S^{*,\rho}$ by $Q^{*,\rho}_p$, this just means that
  $Q_p^{*,\rho} = \{ (z,w)\colon \bar p\in Q_{(\bar  z, \bar w)}^\rho  \}$.
  In the following,
  we will suppress the dependence on $\rho$ from the notation whenever
  we make claims which hold for all admissible parametrizations of a given Segre family.

It is not difficult to see that if $\mathcal S \in \mathfrak{S}_m^\pm$
is a positive
(respectively, negative) $m$-admissible Segre family, then
$\mathcal S^* \in \mathfrak{S}_m^\pm$
is a negative (respectively, positive)
$m$-admissible Segre family, if we are using its
 naturally associated defining function:
 Indeed, to obtain the defining
function $\rho^*(z,\bar\xi,\bar\eta)$ of the general Segre family
$\mathcal S^*$ we need to solve for $w$  in the equation
\begin{equation}\label{dual}
\bar\eta=w e^{\pm iw^{m-1}\left(z\bar\xi+\sum_{k,l\geq
 2}\varphi_{kl}(w)z^k\bar\xi^l\right)} .
\end{equation}
Note that \eqref{dual} implies
\begin{equation}\label{dual2}
w=\bar\eta e^{\mp iw^{m-1}(z\bar\xi+O(z^2\bar\xi^2))}.
\end{equation}
We then obtain from \eqref{dual2}
$w=\rho^*(z,\bar\xi,\bar\eta)=\bar\eta(1+O(z\bar\xi))$.
Substituting the latter representation into \eqref{dual2} gives
$w=\rho^*(z,\bar\xi,\bar\eta)=\bar\eta e^{\mp
i\bar\eta^{m-1}(z\bar\xi+O(z^2\bar\xi^2))}$, as required.

We also need the following Segre family, connected with $S$:
$$\bar{\mathcal S}=\{w=\bar\rho(z,\bar \xi,\bar \eta)\},$$
where for a power
series of the form
$$f(x)=\sum\nolimits_{\alpha \in\mathbb{Z}^d}c_\alpha x^\alpha
$$  we denote by $\bar f(x)$ the series
$\sum_{\alpha \in\mathbb{Z}^d} \bar c_\alpha x^\alpha $. Note that
$\bar{ \mathcal S}$ does not depend on the particular admissible parametrization, in
contrast to the dual family.

 \begin{definition} The Segre family $\bar{\mathcal S}$ is called  the
 \emph{conjugated family} of $\mathcal S$.
 \end{definition}

 If $\sigma: \CC{2}\longrightarrow\CC{2}$ is the antiholomorphic involution $(z,w)\longrightarrow(\bar
 z,\bar w)$, then one simply has $\sigma(Q_p) = \overline {Q_{\sigma(p)}}$.
 We  will denote the Segre variety of a point $p$ with respect to the family $\bar{\mathcal S}$ by
 $\bar Q_p^\rho $. It follows from the definition that if $\mathcal S$ is a
positive (respectively, negative) $m$-admissible Segre family,
then $\bar{\mathcal S}$ is a negative (respectively, positive)
$m$-admissible Segre family.

 In the same manner as for the case of an $m$-admissible Segre
 family, we say that a (general) Segre family $\mathcal S=\{w=\rho(z,\bar \xi,\bar
 \eta)\}$ has a \it real structure, \rm if there exists a smooth real-analytic
 hypersurface $M\subset\CC{2}$, passing through the origin, such
 that $\mathcal S$ is the Segre family of $M$.

 The use of the  dual and the conjugated Segre families is
 illuminated by the fact that

\medskip


 \it A (general) Segre family $\mathcal S$ has a real
 structure if and only if the  conjugated Segre family $\bar{\mathcal S}$ is also
 a dual family, i.e. if there exists an admissible
 parametrization $\rho$  such that $\mathcal S^{*,\rho}=\bar{\mathcal
 S}$ \rm

\medskip

This fact proved, for example, in \cite{divergence} is a corollary of the
reality condition  for a real-analytic hypersurface.

\begin{definition}
Let $\mathcal{ E}_0 \in \mathfrak{E}_m$.
The \emph{$m$-dual}  $\mathcal{ E}_0^* \in \mathfrak{E}_m$ to
$\mathcal{E}_0$ is defined by
$$
\mathcal{ E}_0^* = \mathcal{E} ((\mathcal S^+_m(\mathcal {E}_0))^*).
$$
In other words, $\mathcal{ E}_0^*$ is the $m$-dual of $\mathcal{ E}_0$
if the
negative $m$-admissible Segre family which is dual to the family $\mathcal
S_{\mathcal{ E}_0}^+$ is associated with $\mathcal{ E}_0^*$.

The  \emph{$m$-conjugated  ODE $\overline{\mathcal{ E}_0}$ to $\mathcal{ E}_0$,} is given by
$$
\bar{\mathcal{ E}}_0=  \mathcal{E}(\overline{\mathcal S}^+_m(\mathcal{ E}_0)).
$$
In other words,  $\overline{\mathcal{ E}_0}$ is $m$-conjugated   to $\mathcal{ E}_0$ if
the negative $m$-admissible Segre family which is conjugated to the family
$\mathcal S^+_m(\mathcal E)$ is in fact associated with
$\overline{\mathcal{ E}_0}$.
\end{definition}

The characterization of the existence of a real structure for
Segre families now implies that
\begin{proposition}\Label{exists}\label{exists}
An equation   $\mathcal E \in \mathfrak{E}_m$
has a real structure if and only if
\begin{equation}\Label{Estarbar}\label{Estarbar}
\mathcal E^*=\overline{\mathcal E}
\end{equation}
holds.
\end{proposition}
\begin{proposition}\Label{dualinnf}\label{dualinnf}
Let  $\mathcal S \in \mathfrak{N}_m^\sigma$ be  in normal form. Then both
 $\mathcal S^* $ and $\overline{\mathcal S}$ are in normal form.
\end{proposition}
\begin{proof}
 Let  $$\mathcal S=\left\{w=\bar\eta
e^{i\bar\eta^{m-1}\varphi}\right\},\quad \bar{\mathcal
S}=\left\{w=\bar\eta
e^{-i\bar\eta^{m-1}\tilde\varphi}\right\},\quad \mathcal{ S}^*=\left\{w=\bar\eta e^{-i\bar\eta^{m-1}\varphi^*}\right\}.$$ 
Then, according to Lemma 20 in \cite{nonanalytic},
\begin{gather} \label{phibar}
\tilde\varphi_{kl}(w)=\bar\varphi_{kl}(w),\,k,l\geq 2\\
\label{phistar22}\varphi^*_{22}(w)=\varphi_{22}(w)-i(m-1)w^{m-1},\quad
\varphi^*_{32}(w)=\varphi_{23}(w),\quad \varphi^*_{23}(w)=\varphi_{32}(w),\\
\label{phistar33}
\varphi^*_{33}=\varphi_{33}(w)-\frac{3}{2}(m-1)^2w^{2m-2}-2i(m-1)w^{m-1}\varphi_{22}(w)-iw^m\varphi'_{22}(w). 
\end{gather}
({\bf Important:} we have hereby corrected a typo in \cite{nonanalytic} which is the sign of $\frac{3}{2}$ in the formula \eqref{phistar33}). Now the formula \eqref{phibar} combined with \autoref{nftransfer} immediately implies that the conjugated family $\overline{\mathcal S}$ is in normal form.

In the case $m=1$ the fact that $\mathcal S^*$ is proved similarly. However, for $m>1$ the assertion of the proposition for the dual family requires additional arguments. If we use the normal form conditions and denote
$$A_0(w)=\lambda+mw^{m-1}, \quad B_1(w)=\mu+\nu w^{m-1}+\sigma_1w^{2m-2}+\sigma_2w^{3m-3}$$
(in the notations of \autoref{relation}), then by using \eqref{relations},\eqref{phistar22} and \eqref{phistar33} we can obtain (after a straightforward calculation)
\begin{equation}\Label{equal}\label{equal}
\begin{aligned}
&\varphi_{22}^*=-\frac{i}{2}(m-1)w^{m-1}+\frac{i}{2}\lambda, \\
&\varphi^*_{33}= -\frac{1}{3}(\lambda-(m-1)w^{m-1})^2-\frac{\lambda}{6}(m-1)w^{m-1}-\frac{1}{6}(\mu+\nu w^{m-1}+\sigma_1w^{2m-2}+\sigma_2w^{3m-3}).
\end{aligned}
\end{equation}
Working out now the conditions \eqref{phi2m-2} for the dual family, one can see that they are precisely equivalent to \eqref{equal}, that is why $\mathcal S^*$ is in normal form, as required.
\end{proof}
Further, we  have the following

\begin{convention}
We call two ODEs \eqref{ODE} {\em $\tau$-equivalent}, if there is a (formal or holomorphic) map \eqref{normalmap} with the {\em complex} parameter \eqref{tau} being equal to $\tau$ transforming the two ODEs into each other. The same convention holds for Segre families and real hypersurfaces.
\end{convention}

The next important step is
\begin{proposition}\Label{dualequiv}\label{dualequiv}
If two ODEs $\mathcal E_1,\mathcal E_2$, as in \eqref{ODE}, are $\tau$-equivalent, then  the ODEs $\overline{\mathcal E_1}$, $\overline{\mathcal E_2}$ are $\bar\tau$-equivalent and the ODEs $\mathcal E_1^*,\mathcal E_2^*$ are $\bar\tau$-equivalent.
\end{proposition}
\begin{proof}
For the conjugated ODEs the assertion of the proposition is obvious (one has to simply bar the map between  $\mathcal E_1$ and $\mathcal E_2$). For the dual ODEs, we denote the given map between  $\mathcal E_1$ and $\mathcal E_2$ by $(F(z,w),G(z,w))$. By \autoref{exists}, there exists a map $(\Lambda(\xi,\eta),\Omega(\xi,\eta))$, as in \eqref{normalmap}, such that the product map
$$\bigl(F(z,w),G(z,w),\bar\Lambda(\xi,\eta),\bar\Omega(\xi,\eta)\bigr)$$
maps the associated Segre families $\mathcal S_1,\mathcal S_2$ of respectively
$\mathcal E_1$ and $\mathcal E_2$ into each other. Hence, by definition of the dual family, the product map
$$\bigl(\Lambda(\xi,\eta),\Omega(\xi,\eta),\bar F(z,w),\bar G(z,w)\bigr)$$
maps the associated Segre families into each other. Recall that both Segre families have the form
$$w=\bar\eta+O(z\bar\xi\bar\eta^m).$$
Also recall that $G=O(w),\Omega=O(\eta)$ due to \eqref{normalmap}.
In view of that, writing the transformation rule between the Segre families $\mathcal S_1,\mathcal S_2$ and collecting terms with $z^0\bar\xi^0\bar\eta^j,\,1\leq j\leq m$ gives:
$$G^{(j)}(0,0)=\overline{\Omega^{(j)}(0,0)},$$
which implies the assertion of the proposition.
\end{proof}

We are now in the position to prove the first key proposition leading to \autoref{theor1}.

\begin{proposition}\Label{key}\label{key}

\smallskip

(i) Let $\mathcal E\in \mathfrak{E}_1$ have a real structure, that is,
assume that $\mathcal E$
is associated with a real hypersurface $M\in \mathfrak{R}_1^\pm$.
Let $\mathcal{N}$ be its normal form. Then $\mathcal{N}$ has a real structure, i.e. it is associated with a real hypersurface
$N = N(M) \in \mathcal{R}_1^{\pm}$
Furthermore, $N(M)$ is  biholomorphically  equivalent to $M$ at the origin.

\smallskip

(ii) Let $m>1$ and $\mathcal E \in \mathfrak{E}_m$ have a real structure, that is,
assume that $\mathcal E$ is associated with a real hypersurface
$M \in \mathfrak{R}_m^\pm$ and $\mathcal{N}_\tau$
be its normal form corresponding to some value $\tau\in\RR{}$
of the parameter  \eqref{tau}.
Then $\mathcal{N}_\tau$
has a real structure, i.e. it is associated with a real hypersurface
 $N_\tau (M) \in \mathfrak{R}_m^\pm$.
 Furthermore, $N_\tau (M)$ is  formally  $\tau$-equivalent to $M$
 at the origin.
\end{proposition}
\begin{proof}
Assume that $m>1$. Let $H(z,w)$ be the transformation mapping $\mathcal E$ into $\mathcal{N}_\tau$.
 Then obviously $\bar H$ is a map between
  $\overline{\mathcal E}=\mathcal E^*$
  and $\overline{\mathcal{N_\tau}}$.
   That is, $\bar H$ brings $\overline{\mathcal E}=\mathcal E^*$
   into normal form (which is equal to $\overline{\mathcal{N}_\tau}$),
   and by  \autoref{dualequiv} its parameter \eqref{tau} is
   $\bar \tau = \tau$ as well. Again, by \autoref{dualequiv} we can claim that $\mathcal E^*$ is $\tau$-equivalent to $\mathcal{N}_\tau^*$. Since both  ODEs $\overline{\mathcal{N}_\tau}$ and $\mathcal{N}_\tau^*$ are in normal form, we conclude from the uniqueness of a normalizing map that
$$\overline{\mathcal{N}_\tau}=\mathcal{N}_\tau^*,$$
so that $\mathcal{N}_\tau$ has a real structure.
Hence the Segre families  $\mathcal{S}$ and $\mathcal{S}_\tau$
of the two ODEs $\mathcal E$ and $\mathcal{N}_\tau$ respectively are the Segre families of real hypersurfaces $M$ and $N_\tau (M)$ respectively. By \autoref{exists}, there exists a product map between them. As follows from the above argument, this product map coincides with $(H(z,w),\bar H(\xi,\eta))$, and this implies that $H$ maps $M$ into $N_\tau (M)$, as required for (ii).

In the case $m=1$ the proof is analogous (actually, it is even simpler because the parameter $\tau$ from \eqref{tau} does not occur at all).
\end{proof}
\begin{remark}\Label{onlyreal}\label{onlyreal}
As follows from \eqref{taureal}, when considering normal forms with a real structure for ODEs \eqref{ODE} with a real structure, one needs to
 restrict to $\tau$ being real in \eqref{tau},
 which explains the assumption $\tau\in\RR{}$ in \autoref{key}.
\end{remark}
Our remaining task is to relate the ODE defining function $\Phi$ and the real defining function $h$ {\em in the case when an ODE $\mathcal E$  has a real structure}. For that, we relate the exponential defining function $\varphi$ and the real defining function $h$ by means of the identity
\begin{equation}\Label{phitoh}\label{phitoh}
w=\bar w e^{i\bar w^{m-1}\varphi(z,\bar z,\bar w)},\,\,\,\mbox{when} \,\,\,w=u+\frac{i}{2}u^mh(z,\bar z,u)
\end{equation}
For the terms relevant to the normal form this gives (after successively comparing terms with $z^2\bar z^2$, $z^3\bar z^2$, $z^2\bar z^3$, $z^3\bar z^3$):
\begin{equation}\Label{phih}\label{phih}
\begin{aligned}
&\varphi_{22}=i\frac{m-1}{2}w^{m-1}+h_{22},\\
&\varphi_{32}=h_{32},\quad \varphi_{23}=h_{23},\\
&\varphi_{33}=h_{33}+i(m-1)h_{22}w^{m-1}-\frac{(m-1)(3m-2)}{8}w^{2m-2}+\frac{i}{2}h_{22}'w^m-\frac{1}{12}w^{3m-3}.
\end{aligned}
\end{equation}
Applying now \eqref{relations}, we obtain by a straightforward calculation:
\begin{equation}\Label{Phih}\label{Phih}
\begin{aligned}
&A_0=mw^{m-1}-2ih_{22},\quad B_0=-2h_{23},\quad A_1=-6ih_{32},\\
&B_1=-6h_{33}+8h_{22}^2-iw^mh_{22}'+\frac{1}{4}w^{2m-2}(m-1)(m+2)+\frac{1}{2}w^{3m-3}.
\end{aligned}
\end{equation}

We are now in the position to prove \autoref{theor1}.
\begin{proof}[Proof of \autoref{theor1} and \autoref{theor2}]
In the case $m=1$, we apply \autoref{key} and conclude that $M$ can be mapped biholomorphically onto a real hypersurface \eqref{madmissiblereal} with the desired uniqueness property. The latter hypersurface satisfies \eqref{phi1a}. Now the fact that $M$ is in normal form (i.e. its defining function $h(z,\bar z,u)\in\mathcal N_1$) follows from \eqref{Phih}.

For $m>1$, inspired by \eqref{Phih}, we make the following particular choice of a normal form space $\mathcal D^\sigma_m$:
\begin{equation}\Label{detsigma}\label{detsigma}
\sigma_1:=\frac{1}{4}(m-1)(m+2),\quad \sigma_2:=\frac{1}{2}.
\end{equation}
We then use \autoref{key} and \autoref{onlyreal} to obtain a formal normal form for a hypersurface \eqref{madmissiblereal} (with the above choice of $\sigma$), which is uniquely determined by the real parameter $\tau$, as in \eqref{tau}. Now the fact that $M$ is in normal form (i.e. its defining function $h(z,\bar z,u)\in\mathcal N_m$) again follows from \eqref{Phih}.
\end{proof}
\begin{remark}\Label{manynf}\label{manynf}
The choice of the real parameters $\sigma$ in \eqref{detsigma} and  the respective  normal form space for the real defining function $h$ is of course not the only one possible. However, only the above choice makes the space of normalized power series $h(z,\bar z,u)$ a {\em linear space}, that is why we stay with it.
\end{remark}
\begin{proof}[Proof of \autoref{theor3}]. For the proof of \autoref{theor3}, we combine \autoref{theor2} with an argument similar to the one in \cite{KLS}. Namely, we first identically repeat the argument in \cite{KLS} to turn \eqref{merom} into a first order system of the kind \eqref{mer}. We then    apply Braaksma's Theorem (see Section 2.3) for the resulting system and conclude that the (unique) formal solution of it obtained above (and used to obtain the unique formal normalizing transformation, as in \autoref{theor2}) is $(k_1,...,k_s)$ multi-summable in all but a few directions $d$. As explained in \cite{KLS}, one can pick up such a directions $d^\pm$ that the resulting "large" sectors $S^\pm$ giving the sectorial realizations of the solution contain $R{\pm}$. After substituting the sectorial solutions as the Cauchy data for \eqref{system1}, we finally obtain a sectorial map of $M$ (in the sense of Section 2.3) realizing the formal normalizing map in \autoref{theor2}. It has, furthermore, the multi-summability property. Now the injectivity of the Borel map discussed in Section 2.3 implies the desired uniqueness of the smooth normalizing transformation ({\em if fixing the multi-summability levels $k_j$ and the directions $d^\pm$}). The image $N$ of $M$ under the normalizing transformation is the desired smooth normal form of $M$. The assertion of the theorem follows now from that for \autoref{theor2} and the uniqueness of the sectorial realizations.

\end{proof}

\begin{remark}\label{canonical2}
Following \autoref{canonical}, we shall comment again that   the multi-summability levels $k_j$ and the directions $d^\pm$ are uniquely determined by the collectyion $J$, as in \eqref{7jet}. To show this, one has to inspect the procedure in \cite{braaksma} and make sure that all the data required for deducing $k_j$, $d^\pm$ can be read from the principal matrix \eqref{44}, which is already determined by \eqref{7jet}. We leave the details to the reader.
\end{remark}

\section{Fuchsian type hypersurfaces}

In this section, we introduce the class of Fuchsian type hypersurfaces mentioned in the Introduction, and provide a complete convergent normal form for Fuchsian type hypersurfaces.

\subsection{The Fuchsian type class} We start with
\begin{definition}\Label{fuchsianM}\label{fuchsianM}
A hypersurface \eqref{madmissiblereal} is called {\em a hypersurface of Fuchsian type}, if its  defining function $h(z,\bar z,u)$ satisfies
\begin{equation}\Label{FM}\label{FM}
\begin{aligned}
&\ord\,h_{22}(w)\geq m-1; \,\ord\,h_{23}(w)\geq 2m-2; \,\ord\,h_{33}(w)\geq 2m-2;\\
&\ord\,h_{2l}(w)\geq 2m-l+2, \,\,4\leq l\leq 2m+1;\\
&\ord\,h_{kl}(w)\geq 2m-k-l+5,\,\,k\geq 3,\,l\geq 3,\,7\leq k+l\leq 2m+4.
\end{aligned}
\end{equation}
\end{definition}
 We point out that

 \smallskip

 \begin{itemize}

\item \noindent The Fuchsian condition requires vanishing of an appropriate part of the $(2m+4)$-jet of the defining function $h$ at $0$;

 \smallskip

\item \noindent  It is easy to see from \ref{FM} that {for $m=1$ the Fuchsian type condition holds automatically, while for $m>1$ it fails to hold in general};

 \smallskip

 \item  \noindent  As is shown in the end of this section, the Fuchsian type property is holomorphically invariant.

   \end{itemize}

   \smallskip

\begin{remark}\Label{comparefuchs}\label{comparefuchs}
The property of being Fuchsian extends earlier versions of this property given  respectively in the work  \cite{nonminimalODE} Kossovskiy-Shafikov, and the work  \cite{analytic} of Kossovskiy-Lamel. In the paper \cite{nonminimalODE}, a Fuchsian property of {\em generically spherical} hypersurfaces \eqref{madmissiblereal} was introduced. It is possible to check that for a generically spherical hypersurface the two notions of being Fuchsian coincide. In the paper \cite{analytic}, general hypersurfaces \eqref{madmissiblereal} were considered, but the notion of Fuchsian type considered there is  different from that given in the present paper; it serves to guarantee the regularity of infinitesimal CR-automorphisms, while the property \eqref{FM} guarantees regularity of {\em arbitrary} CR-maps (see next subsection). The property introduced in \cite{analytic} is more appropriately addressed as {\em weak Fuchsian type}, while the property \eqref{FM} as the (actual) Fuchsian type.
\end{remark}

\subsection{Normalization in the Fuchsian type case} Let us introduce, for each $m>1$,  the space $\mathcal N^F_m$ of power series $h(z,\bar z,u)$, as in \eqref{madmissiblereal}, that are of Fuchsian type and  satisfy, in addition,
\begin{equation}\Label{nspaceF}\label{nspaceF}
h_{22}^{(m)}(u)=h_{23}^{(2m-1)}(u)=h_{33}^{(2m-1)}(u)\equiv 0.
\end{equation}
In other words, we have
\begin{equation}\Label{constantsF}\label{constantsF}
h_{22}(u)=const\,\cdot\, u^{m-1}, \quad h_{23}(u)=const\,\cdot\, u^{2m-2}, \quad h_{33}(u)=const\,\cdot\, u^{2m-2}.
\end{equation}

Now, the normalization result for Fuchsian type hypersurfaces is as follows.
\begin{theorem}\Label{theorf}\label{theorf}
Let $M\subset\CC{2}$ be a real-analytic Levi-nonflat hypersurface, which
is of $m$-infinite type at $p\in M$ for some $m>1$ and of Fuchsian type at  $p$, and
assume  that $p$ is a non-resonant point for $M$.  Then there exists a biholomorphic transformation
\begin{equation}\Label{F3}\label{F3}
H:\,(\CC{2},p)\mapsto (\CC{2},0)
\end{equation}
bringing $M$ into the normal form \eqref{nspaceF}. Furthermore, a normalizing transformation $H$ is uniquely determined by the restriction of its differential $dH|_p$ to the complex tangent $T^{\CC{}}_p M$.
\end{theorem}

\begin{corollary}\Label{cor3}\label{cor3}
For a hypersurface $(M,p)$ satisfying the assumptions of
\autoref{theor3}, we have $\dim_\R \mathfrak{aut} (M,p)\leq 2$.
\end{corollary}

\begin{corollary}\Label{cor1}\label{cor33}
Let $H:\,(M,p)\mapsto (M^*,p^*)$ be a local biholomorphism of two hypersurfaces  satisfying the assumptions of
\autoref{theor1}. Then $H$ is uniquely determined by its $1$-jet at the point $p$.
\end{corollary}

\begin{remark}\Label{resonant}\label{resonant}
As can be seen from the description of the resonances in the Fuchsian type case (either $m=1$ or $m>1$), a Fuchsian type hypersurface can admit only finitely many resonances (up to $7$ for $m=1$ and up to $8$ for $m>1$). This means that a resonant normal form (which will still provide a finite-dimensional reduction of the equivalence problem) can be produced and used to solve the equivalence problem in the Fuchsian type case without the non-resonancy assumption. The corresponding construction is, however, a bit technical, and we will not provide it  in this paper.
\end{remark}
We now proceed towards the proof of \autoref{theorf}. As some of the considerations are analogous to that in the case $m=1$, we provide only a brief outline for the respective steps.

First, we translate the Fuchsian type condition for hypersurfaces \eqref{madmissiblereal} described in the introduction onto the language of associated ODEs. For  the functions $\Phi,\Phi^*$, we make use of the expansion
$$\Phi(z,w,\zeta)=\sum_{k\geq 0,l\geq 2}\Phi_{kl}(w)z^k\zeta^l,$$
and similarly for $\Phi^*$. We now introduce
\begin{definition}\Label{fuchsian}\label{fuchsian}
 An ODE $\mathcal{E} \in \mathfrak{E}_m$, defined by
  \eqref{ODE},
  is called {\em Fuchsian} (or {\em a Fuchsian type ODE}), if $\Phi$ satisfies the conditions:
\begin{equation}\Label{FODE}\label{FODE}
\begin{aligned}
&\ord\,\Phi_{02}(w)\geq m-1; \,\ord\,\Phi_{03}(w)\geq 2m-2; \,\ord\,\Phi_{12}(w)\geq m-1;\,\ord\,\Phi_{13}(w)\geq 2m-2;\\
&\ord\,\Phi_{0l}(w)\geq 2m-l+2, \,\,4\leq l\leq 2m+1;\,\,
\ord\,\Phi_{k2}(w)\geq 2m-k, \,\,2\leq k\leq 2m+1; \\
&\ord\,\Phi_{kl}(w)\geq 2m-k-l+3,\,\,k\geq 1,\,l\geq 3,\,5\leq k+l\leq 2m+2.
\end{aligned}
\end{equation}

\end{definition}
We shall then prove
\begin{proposition}\Label{transferfuchs}\label{transferfuchs}
For a Fuchsian type hypersurface $M\subset\CC{2}$, its associated ODE $\mathcal E(M)$ is of Fuchsian type as well.
\begin{proof}
For the terms $\Phi_{02},\Phi_{03},\Phi_{12},\Phi_{13}$ the desired inequalities follow (in both directions) from \eqref{Phih}. For the other terms $\Phi_{kl}$ under consideration, the proof of the desired fact is obtained by a calculation very similar to the one leading to formulas \eqref{Phih}, and we leave the details to the reader.
\end{proof}
\end{proposition}

We will use \autoref{fuchsian} to show the invariancy of the Fuchsian type condition in the end of this section.

For the proof of \autoref{theorf}, we in general follow the scheme in Section 3, and have now to deduce a system of singular ODEs for the Cauchy data $Y(w)$, as in\eqref{Cdata}, relevant to the Fuchsian type situation. For doing so, we have to enforce normalization conditions for the target ODE. Somewhat similarly to the general $m>1$ case, we have to fix a number $\sigma\in\RR{}$ and put
\begin{equation}\Label{partnf}\label{partnf}
\begin{aligned}
&\Phi_{0j2}=0,\,j\geq m;\,\,\Phi_{1j2}=\Phi_{0j3}=0,\,\,  j\geq 2m-1;\\
& \Phi_{1j3}=0,\,\,j\geq 2m-1,\,\,j\neq 3m-3;\,\,\Phi_{1,3m-3,3}=\sigma.
\end{aligned}
\end{equation}
(The other coefficients $\Phi_{0j2},\Phi_{0j3},\Phi_{1j2},\Phi_{1j3}$ we so far leave as free parameters and  show later that their values  are in fact predetermined.) We then collect in \eqref{trule3} terms with
\begin{equation}\Label{0123}\label{0123}
z^kw^j\zeta^l,\,k=0,1,\,l=2,3,\,j\geq 0.
\end{equation}
This gives us a system of four ODEs of the kind \eqref{merom} (with the above discussed parameters involved). For the purposes of this section, we prefer to write dows the obtained system in the form
\begin{equation}\Label{meromF}\label{meromF}
w^{m+1}g_0''=S\bigl(w,Y(w),wY'(w)\bigr),\,\, w^{2m}X''=T\bigl(w,Y(w),wY'(w)\bigr),
\end{equation}
where $X(w):=(g_1(w),f_0(w),f_1(w)),\quad Y(w):=(g_0(w),X(w)),$
and $S,T$ are holomorphic near the origin.
For the functions $T,S$ we will use the expansion
\begin{equation}\Label{expandTS}\label{expandTS}
T(w,Y,\tilde Y)=\sum_{\alpha,\beta\geq 0}T_{\alpha,\beta}(w)Y^\alpha\tilde Y^\beta,
\end{equation}
where $\alpha,\beta$ are multiindices, and similarly for $S$.   We now shall prove the following key
\begin{proposition}\Label{Tkl}\label{Tkl}
Under the Fuchsian type condition, the coefficient functions $T_{\alpha,\beta}(w),S_{\alpha,\beta}(w)$ satisfy
\begin{equation}\Label{Tkle}\label{Tkle}
\ord T_{\alpha,\beta}\geq 2m-1-|\alpha|-|\beta|,\,\,\ord S_{\alpha,\beta}\geq m-|\alpha|-|\beta|,\,\,|\alpha|+|\beta|>0.
\end{equation}
\end{proposition}
\begin{proof}
For the proof, we make use of \eqref{FODE} (applied for the source defining function $\Phi^*$), and then study carefully  the contribution of terms $\Phi^*_{kl}$ into the basic identity \eqref{trule3}. Let us fix for the moment some positive value of $|\alpha|+|\beta|$. Then it is straightforward to check, by considering  \eqref{trule3} and collecting terms  \eqref{0123}, that $T_{\alpha,\beta}$ as above can arise only from $\Phi^*_{kl}$ with $k+l\leq  |\alpha|+|\beta|+4,$ while $S_{\alpha,\beta}$ as above can arise only from $\Phi^*_{kl}$ with $k+l\leq  |\alpha|+|\beta|+2$. (And in the latter cases a respective $\Phi_{kl}^*$ is a factor for $Y^\alpha(wY')^\beta$). Now it is not difficult to verify that \eqref{FODE} implies \eqref{Tkl}.
\end{proof}

\begin{corollary}\Label{lowterms}\label{lowterms}
For the $(0,0)$ coefficient functions in \eqref{meromF} we have
\begin{equation}\Label{00}\label{00}
\ord\,S_{0,0}\geq m;\quad \ord\,T_{0,0}\geq 2m-1.
\end{equation}
As a consequence, for the target ODE defining function $\Phi$ we have:
\begin{equation}\Label{Phi23a}\label{Phi23a}
\begin{aligned}
\Phi_{0j2}=0,&\,\,0\leq j\leq m-2; \\
 \Phi_{1j2}=\Phi_{0j3}=\Phi_{1j3}=0,&\,\,0\leq j\leq 2m-3;\\
\Phi_{0,m-1,2}=\Phi^*_{0,m-1,2};\,\,\Phi_{0,2m-2,3}=\Phi^*_{0,2m-2,3};&\,\,\Phi_{1,2m-2,2}=\Phi^*_{1,2m-2,2};\,\,\Phi_{1,2m-2,3}=\Phi^*_{1,2m-2,3}.
\end{aligned}
\end{equation}
\end{corollary}
\begin{proof}
As follows from the definition of $S_{\alpha,\beta},T_{\alpha,\beta}$ and the conditions \eqref{Tkl}, all terms in the first equation in \eqref{meromF} have order at least $m$ in $w$ with possibly the exception of terms arising from $S_{0,0}$, while all terms in the second equation in \eqref{meromF} have order at least $2m-1$ in $w$ with possibly the exception of terms arising from $T_{0,0}$. This proves \eqref{00}. To prove \eqref{Phi23a}, we note that the $(m-1)$-jet of $S_{0,0}$ and the $(2m-2)$-jet of $T_{0,0}$ respectively are formed from differences  between coefficients $\Phi_{kjl}$ and $\Phi^*_{kjl}$ aparent in \eqref{Phi23a}, and this proves \eqref{Phi23a}.
\end{proof}
\begin{corollary}
The coefficients
\begin{equation}\Label{maincoef}\label{maincoef}
\Phi_{0,m-1,2},\Phi_{0,2m-2,3},\Phi_{1,2m-2,2},\Phi_{1,2m-2,3}.
\end{equation}
are invariants of a Fuchsian type hypersurface under transformations \eqref{specialmap}.
\end{corollary}
\begin{proof}
In the case when the target satisfies \eqref{partnf}, the assertion follows from \eqref{Phi23}. However, even without requiring \eqref{partnf} we similarly have the relations \eqref{00} and hence obtain the desired identities in the last line of \eqref{Phi23}.
\end{proof}

Based on \autoref{lowterms}, we can finally introduce the appropriate normal spaces $\mathcal F^\sigma_m$  for Fuchsian type ODEs: for each $\sigma\in\RR{}$, the space $\mathcal F^\sigma_m$ is the space of all power series $\Phi(z,w,\zeta)$, as in \eqref{ODE}, satisfying \eqref{partnf} and the first two lines in \eqref{Phi23}.  Our immediate goal now is to prove that under an appropriate generic condition (the {\em non-resonancy}), each Fuchsian type ODE \eqref{ODE} can be uniquely mapped, by means of a transformation \eqref{normalmap}, onto another ODE \eqref{ODE} satisfying, in addition,
\begin{equation}\Label{fuchsnfa}\label{fuchsnfa}
\Phi\in \mathcal F^\sigma_m
\end{equation}
(while the coefficients \eqref{maincoef}, again, remain unchanged).
We will in general follow the scheme in the case $m=1$ above; however, the proof of convergence will require some extra arguments.

As discussed above, the normalization condition $\Phi\in \mathcal F^\sigma_m$ supplemented by the condition in the last line of \eqref{Phi23} amount to a system of four ODEs \eqref{meromF} satisfying \eqref{Tkl} {\em with $\alpha,\beta$ being both possibly zero} (as follows from \autoref{lowterms}). Let us consider the latter system \eqref{meromF} in more detail.

\smallskip

\noindent{\bf Step I.} We investigate the existence of formal solutions for \eqref{meromF} vanishing at the origin. Let $$H(w)=\sum_{k\geq 1}H_kw^k$$ be such a formal solution. We substitute such a formal solution into \eqref{meromF} and, for each fixed $k\geq 1$, collect terms of  degree  $m+k-1$ in $w$  in the first equation in \eqref{meromF}, and degree $2m+k-2$ in $w$  in the second equation in \eqref{meromF}, respectively.   Now it follows from \eqref{Tkl} that (i) only the coefficients $h_1,...,h_k$ are present in the resulting identity; (ii) the coefficient $h_k$ comes into the latter identity {\em linearly}. This means that the $k$-th identity can be considered as a linear system in $h_k$, if $h_1,...,h_{k-1}$ are considered as known. If we are able now to prove the nondegeneracy of the latter linear system for each $k\geq 1$ (we address the latter property as the {\em non-resonancy} of a Fuchsian type ODE), then we conclude from the above that a formal solution exists and is unique.

We investigate the nonresonancy condition here very similarly  to the case  $m=1$ (see section 4). Arguments identical to the ones in Section 4 show that the nondegeneracy of an above $k$-th linear system amounts to the nondegeneracy of the matrix
\begin{equation}\Label{mxk}\label{mxk}
\begin{pmatrix}
\alpha_0(k+1-m)-\alpha_0^*+ & -3\beta_0 & -\alpha_1 & -2k\\
k(k+1) &\mbox{} &\mbox{} &\mbox{} \\
\alpha_1(k+1-m)+\alpha_1^* & 3k(k-1+2m-\alpha_0)+ & 2\alpha_2 & \alpha_1\\
\mbox{} & 3(\beta_1-\alpha_0'+m(m-1)) & \mbox{} & \mbox{}\\
2\beta_0(k+1-m) +\beta_0^* & 4\gamma_0 & k(k-1+\alpha_0)+\beta_1 & -\beta_0\\
2\beta_1(k+1-m)+\beta_1^* & 4\gamma_1 & k\alpha_1+2\beta_2 & k(k-1+\alpha_0)
\end{pmatrix}
\end{equation}
Here, in the notations \eqref{ABC} and in contrast with \eqref{expandABC}, the constants $\alpha_j,\alpha_j^*,\beta_j,\beta_j^*,\gamma_j$ are
\begin{equation}\Label{coeffs}\label{coeffs}
\begin{aligned}
&\alpha_0:=\frac{1}{(m-1)!}A_0^{(m-1)},\,\,\alpha^*_0:=\frac{1}{(m-1)!}A_0^{(m)},\\
&\alpha_j:=\frac{1}{(2m-2)!}A_j^{(2m-2)},\,\,\alpha^*_j:=\frac{1}{(2m-2)!}A_j^{(2m-1)},\,\,j=1,2\\
&\beta_j:=\frac{1}{(2m-2)!}B_j^{(2m-2)},\,\,
\beta_j^*:=\frac{1}{(2m-2)!}B_j^{(2m-1)},\,\,\gamma_j:=\frac{1}{(2m-2)!}C_j^{(2m-2)},\quad j=0,1,2.
\end{aligned}
\end{equation}
Accordingly we give the following
\begin{definition}\Label{Fuchsnonres}\label{Fuchsnonres}
A Fuchsian type ODE is called {\em nonresonant}, if the associated matrices \eqref{mxk} are nondegenerate for each integer $k\geq 1$. A Fuchsian type hypersurface \eqref{madmissiblereal} is called {\em nonresonant} at the origin, if its associated ODE \eqref{ODE} is nonresonant.
\end{definition}
\begin{remark}\Label{finiteresF}\label{finiteresF}
There can exist only finitely many resonances for a hypersurface \eqref{madmissiblereal} in the Fuchsian type case with $m>1$ (in fact, at most $8$ of them).
\end{remark}

\begin{proposition}\Label{nontrivialF}\label{nontrivialF}
There exist Fuchsian type hypersurfaces \eqref{madmissiblereal} which are non-resonant at $0$. Accordingly, a generic Fuchsian type hypersurface \eqref{madmissiblereal} (in the sense of the jet topology in the space of defining functions $h(z,\bar z,u)$) is non-resonant at $0$.
\end{proposition}
\begin{proof}
The proof is very analogous to that of \autoref{nontrivial} and\autoref{nontrivialm}, that is why we leave its details to the reader.
\end{proof}

The above arguments prove that, {\em for any Fuchsian type ODE \eqref{ODE} there exists a unique formal transformation \eqref{specialmap} bringing it to a normal form.}

\smallskip

\noindent{\bf Step II.} It remains to deal with the convergence of the normalizing transformation. Since the formal transformation under discussion arises as a solution of the regular Cauchy problem \eqref{system1}, it is sufficient to prove the convergence of a formal solution of \eqref{meromF}.

Let $H(w)$ be such a formal solution. We decompose it as
\begin{equation}\Label{trick}\label{trick}
H(w)=P(w)+Z(w),
\end{equation}
where $P(w)$ is a  polynomial without constant term of degree $\leq 2m-1$, 
while where $Z(w)$ is a formal series
of the kind $O(w^{2m})$.
The substitution \eqref{trick} (for a fixed $(P(w)$) turns \eqref{meromF} into a similar system of ODEs for the unknown function $Z(w)$. We shall now prove
\begin{lemma}\Label{0110}\label{0110}
The transformed  system (in the same way as the initial system)  satisfies
\begin{equation}\Label{0110a}\label{0110a}
\ord\,\tilde S_{01}\geq m-1,\,\,\ord\,\tilde S_{10}\geq m-1,\,\,\ord\,\tilde T_{01}\geq 2m-2,\,\, \ord\,\tilde T_{10}\geq 2m-2
\end{equation}
(the tilde here stands for coefficients of the transformed system).
\end{lemma}
\begin{proof}
the proof of the lemma is obtained by putting together the expansion \eqref{expandTS}, the conditions \eqref{Tkl}, and the fact that $P(w)$ is vanishing at the origin.
\end{proof}
Now, based on \autoref{0110}, we perform the substitution
\begin{equation}\Label{ZtoU}\label{ZtoU}
Z:=w^{2m}U,
\end{equation}
which turns the "tilde" system into a new system of four meromorphic ODEs for the unknown function $U$, which, according to \eqref{trick}, has a formal solution $U(w)$ vanishing at the origin. It is straightforward to check then, by combining \eqref{ZtoU} and \eqref{0110a}, that the new system system can be written in the form
\begin{equation}\Label{bb2}\label{bb2}
w^2U'=R(w,U,wU'),
\end{equation}
where $R$ is a holomorphic  function defined near the origin. Performing finally in the standard fashion the substitution
$$V:=wU'$$
and introducing the extended vector function $\bold U:=(U,V)$, we obtain a first order ODE
\begin{equation}\Label{bb1}\label{bb1}
w\bold U'=Q(w,\bold U'),
\end{equation}
where $Q$ is a holomorphic near the origin function. The ODE \eqref{bb1} is a Briot-Bouquet type ODE (see Section 2), hence its formal solutions are convergent, as required.

By competing Steps I and II, we have proved
\begin{theorem}\Label{fuchsnf}\label{fuchsnf}
Fro any $\sigma\in\RR{}$, a nonresonant Fuchsian type ODE \eqref{ODE} can be brought to a normal form $\Phi\in\mathcal F^\sigma_m$  by a transformation \eqref{normalmap}. A normalizing transformation is defined uniquely.
\end{theorem}

\begin{proof}[Proof of \autoref{theorf}] The way how \autoref{fuchsnf} implies \autoref{theor3} is very analogous to the argument in Section 6, and we leave the details to the reader. We just clarify that the particular choice of the parameter $\sigma$ which is suitable for achieving the normalization conditions \eqref{nspaceF} is
$$\sigma=\frac{1}{2}$$
(this choice is illuminated by the transfer formula \eqref{Phih}).
\end{proof}

In the end of the section we would like to prove the {\em invariance} of the Fuchsian type condition.
\begin{theorem}\Label{fuchsinv}\label{fuchsinv}
The property of being Fuchsian for a hypersurface \eqref{madmissiblereal} does not depend on the choice of (formal or holomophic) coordinates of the kind \eqref{madmissiblereal}.
\end{theorem}
\begin{proof}
In view of \autoref{fuchsian}, we can switch to associated ODEs and it is enough to prove the invariance of the Fuchsianity for them. As discussed above, we can restrict to transformations \eqref{normalmap}. Let us consider then the transformation rule \eqref{trule3} (with a fixed transformation within it), when the source ODE (with the defining function $\Phi^*$) is of Fuchsian type. We then claim the following: for all the coefficient functions $\Phi_{kl},\,\,k\geq 0,\,\,l\geq 2$ involved in the Fuchsianity conditions \eqref{FODE}, {\em with the exception of the coefficients functions} $\Phi_{k2},\Phi_{k2}^*,\,\,k\geq 2$, the Fuchsian conditions \eqref{FODE} are satisfied.
 Indeed, we fix any $(k,l)$ relevant to \eqref{FODE}, and from the transformation rule \eqref{trule3} we can see that the target coefficient function $\Phi_{kl}$ is a sum of three groups of terms: (i) terms $\Phi^*_{\alpha\beta}$ with $\alpha+\beta\geq k+l$ which are multiplied by a power series in $w$ with order at $0$ at least $k+l-\alpha-\beta$; (ii) terms $\Phi^*_{\alpha\beta}$ with $\alpha+\beta< k+l$; (iii) terms arising from the expressions $I_j,\,0\leq j\leq 3$ (relevant for $l=2,3$ only). In view of the linearity of the Fuchsianity conditions in $k,l$, it is not difficult to see that terms of the first kind all have order at $0$ at least as the one required for the Fuchsianity. Terms of the second kind already all have order bigger than the one required for Fuchsianity. Finally, terms of the third kind automatically provide order at least $2m$ sufficient for the Fuchsianity, except for the case $l=2$. For $k=0,1$ and $l=2$ though even the automatically provided order $m$ suffies, and this proves the claim.

It remains to deal with terms $\Phi_{k2}$ with $k\geq 2,\,2\leq k\leq 2m+1$. We note, however, that the ODEs under consideration have a real structure, that is why (in view of the reality condition) we have
\begin{equation}\Label{heq}\label{heq}
\ord\,h_{kl}(w)=\ord\,h_{lk}(w)
\end{equation}
for all $k,l$. This, in view of the transfer relations between $\Phi$ and $h$ (similar to \eqref{phitoh}) gives, in particular:
$$\ord\,\Phi_{k2}(w)=\ord\,h_{k+2,2}(w)=\ord\,h_{2,k+2}(w)=\ord\,\Phi_{0,k+2}(w)\geq 2m-k$$
(the last inequality follows from the Fuchsianity condition for $\Phi_{0,k+2}$ being already proved). This finally proves the theorem

\end{proof}

\bibliographystyle{plain}
 \bibliography{bibfile}









\end{document}